\documentclass[11pt,a4paper]{article}
\usepackage{cite}
\usepackage{amsmath, amsthm, amssymb,tikz}
\usepackage{tikz-cd}
\usetikzlibrary{arrows}
\usepackage{fancyhdr}
\parskip=\baselineskip
\font\teneurm=eurm10 \font\seveneurm=eurm7 \font\fiveeurm=eurm5
\newfam\eurmfam
\textfont\eurmfam=\teneurm \scriptfont\eurmfam=\seveneurm
\scriptscriptfont\eurmfam=\fiveeurm

\font\teneusm=eusm10 \font\seveneusm=eusm7 \font\fiveeusm=eusm5
\newfam\eusmfam
\textfont\eusmfam=\teneusm \scriptfont\eusmfam=\seveneusm
\scriptscriptfont\eusmfam=\fiveeusm

\font\tencmmib=cmmib10 \skewchar\tencmmib='177
\font\sevencmmib=cmmib7 \skewchar\sevencmmib='177
\font\fivecmmib=cmmib5 \skewchar\fivecmmib='177
\newfam\cmmibfam
\textfont\cmmibfam=\tencmmib \scriptfont\cmmibfam=\sevencmmib
\scriptscriptfont\cmmibfam=\fivecmmib

\RequirePackage{color}
\definecolor{myred}{rgb}{0.75,0,0}
\definecolor{mygreen}{rgb}{0,0.5,0}
\definecolor{myblue}{rgb}{0,0,0.65}

  \RequirePackage[pdftex,
   colorlinks = true,
   urlcolor = myblue, 
   citecolor = mygreen, 
   linkcolor = myred, 
   ]
 {hyperref}
\numberwithin{equation}{section}
\newtheorem{lemma}{Lemma}[section]
\newtheorem{theorem}[lemma]{Theorem}
\newtheorem{conjecture}[lemma]{Conjecture}
\newtheorem{corollary}[lemma]{Corollary}
\newtheorem{proposition}[lemma]{Proposition}
\theoremstyle{definition}
\newtheorem{example}[lemma]{Example}
\newtheorem{remark}[lemma]{Remark}
\newtheorem{definition}[lemma]{Definition}

\def\16{{\bf 16}}
\def\1{{\bf 1}}
\def\2{{\bf 2}}
\def\3{{\bf 3}}
\def\4{{\bf 4}}
\def\bar{\overline}

\def\hat{\widehat}
\def\Sp{{\mathrm{Sp}}}

\def\tSp{{\widetilde{\mathrm{Sp}}}}

\def\Sym{{\mathrm{Sym}}}

\DeclareMathOperator{\Hilb}{Hilb}
\DeclareMathOperator{\Spec}{Spec}

\DeclareMathOperator{\Proj}{\mathrm{Proj}}

\DeclareMathOperator{\Lie}{Lie}
\DeclareMathOperator{\Hom}{Hom}
\DeclareMathOperator{\Fl}{Fl}
\DeclareMathOperator{\Gr}{Gr}
\DeclareMathOperator{\Jac}{Jac}

\DeclareMathOperator{\Pic}{Pic}

\DeclareMathOperator{\Res}{Res}
\DeclareMathOperator{\cox}{\textbf{cox}}
\DeclareMathOperator{\HY}{HY}
\DeclareMathOperator{\FT}{FT}

\DeclareMathOperator{\diag}{\mathrm{diag}}
\DeclareMathOperator{\gr}{gr}
\DeclareMathOperator{\HHH}{\mathrm{HHH}}

\newcommand{\tW}{\widetilde{W}}

\newcommand{\cO}{\mathcal{O}}

\newcommand{\cP}{\mathcal{P}}

\newcommand{\cC}{\mathcal{C}}
\newcommand{\cI}{\mathcal{I}}
\newcommand{\cN}{\mathcal{N}}
\newcommand{\cT}{\mathcal{T}}
\newcommand{\cK}{\mathcal{K}}
\newcommand{\cV}{\mathcal{V}}

\newcommand{\fb}{\mathfrak{b}}

\newcommand{\tT}{\widetilde{T}}
\newcommand{\fl}{\mathfrak{l}}

\newcommand{\fg}{\mathfrak{g}}
\newcommand{\ft}{\mathfrak{t}}
\newcommand{\bPic}{\overline{\Pic}}
\newcommand{\bJac}{\overline{\Jac}}
\newcommand{\bx}{\mathbf{x}}
\newcommand{\by}{\mathbf{y}}
\newcommand{\bI}{\mathbf{I}}

\newcommand{\Sn}{\mathfrak{S}_n}
\newcommand{\Br}{\mathfrak{B}r}
\newcommand{\C}{{\mathbb C}}

\newcommand{\N}{{\mathbb N}}
\newcommand{\A}{\mathbb{A}}
\newcommand{\G}{\mathbb{G}}
\newcommand{\Z}{{\mathbb Z}}
\newcommand{\Q}{{\mathbb Q}}

\newcommand{\tJ}{\widetilde{J}}
\renewcommand{\L}{\mathbb{L}}
\renewcommand{\hom}{\mathrm{Hom}}

\renewcommand{\P}{\mathbb{P}}
\begin{document}
\begin{titlepage}
\begin{flushright}
\end{flushright}
\vskip 1.5in
\begin{center}
{\bf\Large{Unramified affine Springer fibers and isospectral Hilbert schemes}}
\vskip
0.5cm {Oscar Kivinen} \vskip 0.05in {\small{\textit{Department of Mathematics, California Institute of Technology}}}
\end{center}
\vskip 0.5in
\baselineskip 16pt
\begin{abstract}
For any connected reductive group $G$ over $\C$, we revisit 
Goresky--Kottwitz--MacPherson's description of the torus 
equivariant Borel-Moore homology of affine Springer fibers 
$\Sp_{\gamma}\subset \Gr_G$, where $\gamma=zt^d$ and $z$ is 
a regular semisimple element in the Lie algebra of $G$. In the case $G=GL_n$, we relate the 
equivariant cohomology of $\Sp_\gamma$ to Haiman's 
work on the isospectral Hilbert scheme of points on the plane. We also explain the connection to the
HOMFLY homology of $(n,dn)$-torus links, and formulate a conjecture
describing the homology of the Hilbert scheme of points on the curve $\{x^n=y^{dn}\}$.  
\end{abstract}
\date{\today}
\tableofcontents
\end{titlepage}
\pagestyle{fancy}
\lhead{Affine Springer fibers and Hilbert schemes}
\rhead{}

\section{Introduction}
In this paper, we study a family of affine Springer fibers depending on 
a connected reductive group $G$ over $\C$ and a positive integer $d$. 
Recall that an affine Springer fiber $\Sp_\gamma^{\textbf{P}}$ is a 
sub-ind-scheme of a partial affine flag variety $\Fl^\textbf{P}$ (see 
\cite{Yun} and Section \ref{sec:affinespringer}) that can be 
informally thought of as a zero-set of a vector field for an element of 
the loop Lie algebra of $G$, $\gamma \in \fg\otimes \C((t))$. For us, $
\gamma=zt^d$, where $z$ is any regular semisimple element in $\fg(\C)$.
Without loss of generality, we may take $z$ to be an element of $\Lie(T)^{reg}$, where $T$ is a fixed maximal torus of $G$. In fact, all of our results hold for $\gamma\in \Lie(T)^{reg}\otimes \C((t))$ that are equivalued, but for simplicity we only consider this case.

Using the methods of Goresky-Kottwitz-MacPherson \cite{GKM1,GKM2}, 
we compute the equivariant Borel-Moore homology of $\Sp_\gamma^\textbf{P}$ when $
\textbf{P}$ is a maximal compact subgroup. In 
this case, we simply denote $\Sp_\gamma^\textbf{P}=\Sp_\gamma$. This is by definition a 
reduced sub-ind-scheme of the affine Grassmannian of $G$. Fix a maximal torus and a Borel subgroup
$T\subset B\subset G$, and denote $\Lie(T)=\ft, \Lie(B)=\fb, \Lie(G)=
\fg$. Let moreover the cocharacter lattice of $T$ be $\Lambda:=X_*(T)$. Denote by $\C[\Lambda]=
\C[X_*(T)]$ the group algebra of the cocharacter lattice. This can be canonically identified with functions on the Langlands dual torus $T^\vee$, or as the (non-quantized) $3d$ $\cN=4$ Coulomb branch algebra for $(T,0)$ as in \cite{BFN2}.

Our first result is the following theorem, proved as 
Theorem \ref{thm:mainthm}.
\begin{theorem}
\label{thm:intromainthm}
Let $\Delta=\prod_\alpha y_\alpha \in H^*_T(pt)$ be the Vandermonde element. The equivariant Borel-Moore homology of $X_d:=\Sp_{t^dz}$ for a reductive group $G$ is up to multiplication by $\Delta^d$ canonically isomorphic as a (graded) $\C[\Lambda]\otimes\C[\ft]$-module to the ideal
$$J_G^{(d)}=\bigcap_{\alpha\in \Phi^+} J_\alpha^d \subset \C[\Lambda]\otimes \C[\ft].$$ In particular, there is a natural algebra structure on $\Delta^d H_*^T(\Sp_\gamma)$ inherited from $\C[\Lambda]\otimes \C[\ft]$, and $J_G^{(d)}$ is a free module over $\C[\ft]$.
\end{theorem}
Throughout, $H_*^T(-)$ denotes the equivariant BM homology, see Section \ref{sec:BMhomology} for details. In a few places, we also use the ordinary $T$-equivariant homology as in \cite{GKM2}; it is denoted $H_{*,ord}^T(-)$.

\subsection{Anti-invariants and subspace arrangements}
Let $W$ be the finite Weyl group associated with $G$ and $sgn$ be 
the one-dimensional representation of $W$ where all reflections act by 
$-1$. 
Observe that there is a natural left action $W\times T\to T$, and therefore actions $$W\times T^*T^\vee\to T^*T^\vee, W\times \C[T^*T^\vee]\to \C[T^*T^\vee].$$
Note that the cocharacter lattice $\Lambda=X_*(T)$ naturally identifies with the character lattice of $T^\vee$. In particular, $\C[\Lambda]\cong \C[T^\vee]$, where the left-hand side denotes group algebra and the right-hand side denotes ring of regular functions. The cotangent bundle of $T^\vee$ is trivial, and in particular has fibers $\ft$. Therefore $\C[\Lambda]\otimes \C[\ft]\cong \C[T^*T^\vee]$.
 
Using the description of the equivariant Borel-Moore homology given in Theorem \ref{thm:intromainthm}, we expect a relationship between
the cohomology of $\Sp_\gamma$ and the $sgn$-isotypic component of the natural diagonal $W$-action on $
\C[T^*T^\vee]$. First of all, it is not hard to see the following result.
\begin{theorem}
\label{thm:introanti}
Let $I_G\subseteq \C[T^*T^\vee]$ be the ideal generated by $W$-alternating regular functions in $\C[T^*T^\vee]$ with respect to the diagonal action. Then there is an injective map
$$I_G^d\hookrightarrow J_G^{(d)}\cong \Delta^d H^T_*(\Sp_\gamma).$$
\end{theorem}

Consequently, any $W$-alternating regular function on $T^*T^\vee$ has a unique expression as a cohomology class in $H_*^T(\Sp_\gamma)$, where $\gamma=zt$.

In the case when $G=GL_n$, this isotypic part for the 
corresponding action on $T^*\ft^\vee$ was studied by Haiman \cite{Hai} 
in his study of the Hilbert scheme of points on the plane. 
More specifically, he considered the ideal $I\subset \C[x_1,\ldots, 
x_n,y_1,\ldots, y_n]$ generated by the anti-invariant polynomials, and 
proved that it is first of all equal to $J=\bigcap_{i\neq j}\langle x_i-x_j,y_i-
y_j\rangle$ and moreover free over the $y$-variables. Note that if 
$f\in \C[\bx^\pm, \by]$, it is by definition of the form $f=\frac{g}{(x_1\cdots x_n)^k}$ for some $g\in \C[\bx,\by]$ and $k\geq 0$. Since 
the denominator is a symmetric polynomial, $g\in \C[\bx,\by]$ is 
alternating for the diagonal $\Sn$-action if and only if $f$ is so. In 
particular, in the localization $\C[\bx^\pm,\by]$ we have that  $I_\bx
\cong I_{GL_n}$ for $I_G$ as in Theorem \ref{thm:introanti}.

Let us quickly sketch how the Hilbert scheme of points $\Hilb^n(\C^2)$ 
enters the picture. Let $A\subset \C[\bx,\by]$ be the space of 
antisymmetric polynomials for the diagonal action of $\Sn$. From for 
example \cite[Proposition 2.6]{HaiMacdonald}, we have that $$\Proj 
\bigoplus_{m\geq 0} A^m\cong \Hilb^n(\C^2).$$ In addition, 
$$\Proj \bigoplus_{m\geq 0} I^m \cong X_n,$$ where $$X_n\cong (\C^{2n}\times_{\C^{2n}/\Sn} \Hilb^n(\C^2))^{red}$$ is the so-called {\em isospectral Hilbert scheme}. The superscript $red$ means that we are taking the reduced fiber product, or fiber product in category of varieties instead of schemes.

By results of \cite{Hai1}, we have $I^m=\bigcap_{i\neq j} \langle x_i-x_j, y_i-y_j\rangle^m$, so that $I^d_\bx\cong J_{GL_n}^{(d)}$. In Section \ref{sec:isospectral}, we prove our next main result following this line of ideas.
\begin{theorem}
\label{thm:introisospectral}
There is a graded algebra structure on $$\bigoplus_{d\geq 0} \Delta^d H_*^T(\Sp_{zt^d}).$$ When $G=GL_n$, we have
$$\Proj \bigoplus_{d\geq 0} \Delta^d H_*^T(\Sp_{zt^d}) \cong Y_n,$$ where $Y_n$ is the isospectral Hilbert scheme on $\C^*\times \C$.
\end{theorem}
We next observe that the natural map $\rho: X_n\to \Hilb^n(\C^2)$ 
restricts to a map $Y_n\to \Hilb^n(\C^*\times \C)$. 
Define the {\em Procesi bundle} on $\Hilb^n(\C^2)$ to be $\cP:=\rho_*
\cO_{X_n}$. By results of Haiman, this is a vector bundle of rank $n!$. 
We then have the following corollary to Theorem 
\ref{thm:introisospectral}.
\begin{corollary}
We have that 
$$H^0(\Hilb^n(\C^*\times \C, \cP\otimes \cO(d))=J^{(d)}_{GL_n}=\Delta^d\cdot H_*^T(\Sp_\gamma),$$ where $\gamma=zt^d$.
\end{corollary}

Our results can be at least interpreted in terms of 
the Coxeter arrangement for the root data of $G$ or $G^\vee$. More precisely, $
\C[X_*(T)]$ can be thought of as the ring of functions on the dual 
torus $T^\vee\cong (\C^*)^n$, which in turn is the complement of ``coordinate 
hyperplanes" in $\ft^\vee\cong X_*(T)\otimes_\Z \C$ for the basis given by 
fundamental weights determined by $B$. Note that the resulting divisor is independent of $B$. 

There is another hyperplane arrangement in this 
space, determined by $\Phi^\vee$, which is called the Coxeter 
arrangement, and can be viewed as the locus where at least one of the positive coroots vanishes. Inside $T^\vee$, this corresponds to 
the divisor $$\cV=\bigcup_{\alpha} \cV_\alpha = \left\lbrace \prod_{\alpha\in \Phi^+} (1-x^{\alpha^\vee})=0 \right\rbrace\subset T^\vee.$$ 

Let us go back to $\ft^\vee$ for a while. We may ``double" the Coxeter hyperplane arrangement inside $\ft^\vee$ to a codimension two arrangement in $\ft\oplus \ft^\vee$ as follows. 
Each $\alpha^\vee$ corresponds to a positive root $\alpha$ for $G$, whose vanishing locus is a hyperplane $\cV_\alpha^\vee$ in $\ft$. Both $\alpha, \alpha^\vee$ also determine hyperplanes inside $\ft\oplus \ft^\vee$ by the same vanishing conditions, and by abuse of notation we will denote these also by $\cV_\alpha, \cV^\vee_\alpha$. By intersecting, we then get a codimension two subspace $\cV_\alpha\cap \cV_\alpha^\vee$. It is clear from the description that the union of these subspaces as $\alpha$ runs over $\Phi^+$ is defined by the ideal
$$\bigcap_{\alpha \in \Phi^+} \langle y_\alpha, x_{\alpha^\vee}\rangle \subset \C[\ft\oplus \ft^\vee].$$

Here $x_{\alpha^\vee}$ and $y_\alpha$ are the linear functionals associated to $\alpha^\vee, \alpha$.
Localizing away from the coordinate hyperplanes in $\ft^\vee$, we then see that the ideal $J_G\subset \C[T^*T^\vee]$ from earlier determines a doubled Coxeter arrangement inside $T^*T^\vee$. In fact, it is immediate from the description that its Zariski closure inside $T^*\ft^\vee$ equals $\bigcup_\alpha \cV_\alpha\cap \cV_\alpha^\vee$.
In the $GL_n$ case, this doubled subspace arrangement coincides with the one studied by Haiman. In 
\cite[Problem 1.5(b)]{HaiComm}, Haiman poses the question of what happens for other root systems.  Reinterpreting the doubling procedure to mean the root system and its (Langlands) dual in $T^*T^\vee$, instead of taking $\cV\otimes \C^2\subset \ft\otimes \C^2$, we have  freeness of $J_G$ in ``half of the variables" by Theorem \ref{thm:intromainthm}, which answers the question in {\it loc. cit}.

There are several other corollaries to Theorem \ref{thm:intromainthm} that we  now illustrate. 

Let $G=GL_n$. It is a conjecture of Bezrukavnikov-Qi-Shan-Vasserot (private communication) that under the lattice action of $\Lambda$ on 
$H^*(\tSp_\gamma)$, where $\gamma=zt$, we also have 
$$H^*(\tSp_\gamma)^\Lambda\cong DH_n$$ and $$H^*(\Sp_\gamma)^\Lambda\cong DH_n^{sgn}.$$
While we are not able to prove said conjecture, we are able to prove an analogous statement in Borel-Moore homology for the {\em coinvariants} under the lattice action on the sign character part, see Theorem \ref{thm:bezconj}. (From this, one can also deduce the statement in cohomology for the sign character part.)

\begin{theorem}
We have $$H_*(\Sp_\gamma)_\Lambda\cong DH_n^{sgn}.$$
\end{theorem}

Let us then discuss the freeness over $\Sym(\ft)$ of the ideals $J^{(d)}_G$ and related ideals in more detail. For example, 
in type A, 
it is clear that the simultaneous substitution $x_i\mapsto x_i+c, c\in \C, i=1,\ldots, n$ leaves $J_G$ invariant, so that the freeness over $\Sym(\ft)$ of $\bigcap_{i\neq j} \langle x_i-x_j,y_i-y_j\rangle \subset \C[\bx,\by]$ can be deduced from that of $J_G$. We remark that the results of Section \ref{sec:ratell} can also be used to show this statement.

\begin{theorem}
\label{thm:freeness}
Let $G=GL_n$ and $J=\bigcap_{i\neq j} \langle x_i-x_j,y_i-y_j\rangle \subset \C[\bx,\by]$.
Then we have $\Delta^d\cdot H_*^T(\Sp_\gamma)\cong J^d_\bx\subset \C[\bx^
\pm,\by]$, where the subscript $\bx$ denotes localization in the $\bx$-variables. In particular, $J^d\subset \C[\bx,\by]$ is free over $\C[\by]:=
\C[y_1,\ldots,y_n]$. 
\end{theorem}

It is somewhat subtle that Theorem \ref{thm:intromainthm} does not 
immediately imply the freeness over $\Sym(\ft)$ of the ideals in $\C[T^*T^\vee], \C[T^*\ft^\vee]$ generated by the anti-invariants, even in type A. Of course, 
one would hope for a similar description as Haiman's for arbitary $G$, but it seems likely some modifications are in order outside of type $A$ \cite{Gor, Gin}.

Haiman's original proof \cite{Hai1} of a related stronger statement, 
``the Polygraph Theorem", implying the freeness of the anti-invariant ideal $I$ and its powers over $\C[\by]$, 
and thus freeness of $J^d=J^{(d)}$ over $\C[\by]$, involves rather subtle commutative 
algebra. Until recently, it has been the only way of showing the freeness of 
$J^{(d)}$ without giving a clear conceptual explanation. On the other hand, 
Theorem \ref{thm:freeness} gives a quite hands-on explanation of this phenomenon. It does not seem to be impossible to use the representation-theoretic interpretation of $J^{(d)}$ and the $S_n$-action on $H_*^T(\Sp_\gamma)$ to try to directly attack freeness of $I^d$.

In fact, recent work of Gorsky-Hogancamp \cite{GoHo} on knot homology gives 
another proof of Theorem \ref{thm:freeness}. Their results also rest on 
results of Elias-Hogancamp \cite{EH} on the HOMFLY homology of $(n,dn)
$-torus links, which involves some quite nontrivial computations with 
Soergel bimodules. In this paper, the complexity of the freeness 
statement is hidden in the cohomological purity of $\Sp_\gamma$ as proved by Goresky-
Kottwitz-MacPherson \cite{GKM3}.

\subsection{Relation to braids}

Let us first consider a general connected reductive group $G$. {\em Any} $\gamma\in \fg\otimes \C((t))$ gives a nonconstant (polynomial) loop $[\gamma] \in \hom(\Spec\C[t^\pm], \ft^{reg}/W),$ through which we get a conjugacy class $\beta\in\pi_1(\ft^{reg}/W)\cong \Br_W$. Note that we do not have a natural choice of basepoint, so that $\beta$ is not a bona fide element of the braid group, but just a conjugacy class. 

Let now $G=GL_n$. Then the braid closure  $\overline{\beta}$ is a knot or link in $S^3$.
For links in the three-sphere, it is natural to consider various link invariants, such as the triply graded Khovanov-Rozansky homology (or HOMFLY homology) \cite{KR2}. This is an assignment 
$$\beta\mapsto \HHH(\overline{\beta})$$ of $\Z^{\oplus 3}$-graded $\Q$-vector spaces to braids, which factors through Markov equivalence. The invariant $\HHH(-)$ was recently generalized to $y$-ified HOMFLY homology in \cite{GoHo}. It is an assignment of $\Z^{\oplus 3}$-graded $\C[y_1,\ldots, y_m]$-modules to braids, and has many remarkable properties. We will discuss these in more detail in Section \ref{sec:braids}.
 
We are mostly interested in $\HY(-)$ for the braid associated to $\gamma=zt^d$, following previous parts of this introduction. In this case, $\beta$ is the $(nd)$th power of a Coxeter braid $\cox_n$ (positive lift of the Coxeter element in $S_n$). In particular, $\beta$ is the $(d)$th power of the {\em full twist} braid $\cox^n_n$. Note that since $\beta$ is central, it is alone in its conjugacy class and thus an actual braid. Taking the braid closure of $\beta$, it is well-known that we recover the $(n,dn)$ torus link $T(n,dn)$.

\begin{remark}
The closures of powers of the Coxeter braids $\cox^m_G$ and their relation to 
affine Springer theory has appeared in the literature in several places \cite{OY1, VV1, GORS}, in the case where $m$ is prime to the Coxeter number of $G$. The case we consider is the one where $m$ is a multiple of the Coxeter number.
\end{remark}

Now, progress in knot homology theory by several people \cite{GoHo, GNR, EH, Mel} has lead to an identification of the Hochschild degree zero part 
of the  $y$-ified HOMFLY homology of $(n,nd)$-torus links and the 
ideals $J^d=\bigcap_{i<j}\langle x_i-x_j,y_i-y_j\rangle$ from above. In particular, combining these results and Theorem \ref{thm:mainthm}, we get the following corollary, proved in Corollary \ref{coro:ft}.
\begin{corollary}
There is an isomorphism of $\C[\bx^\pm, \by]$-modules 
$$\Delta^dH_*^T(\Sp_\gamma)\cong \HY(\FT_n^d)^{a=0}\otimes_{\C[\bx]}\C[\bx^\pm]$$ for $\gamma=zt^d$.
\end{corollary}

\begin{remark}
Assuming the purity of affine Springer fibers, one is able to deduce further results. If $$\gamma=\begin{pmatrix}
a_1t^{d_1}&&\\ & \ddots & \\ &&a_nt^{d_n}
\end{pmatrix},$$ the construction above gives us a pure braid $\beta$ whose braid closure has linking numbers $d_{ij}=\min(d_i,d_j)$ between components $i, j$.

By \cite[Proposition 5.5]{GoHo},
if $\beta$ has "parity", ie. $\HHH(\bar{\beta})$ is only supported in even or odd homological degrees, we have the following isomorphism of bigraded $\C[\bx, \by]$-modules
$$\HY^{a=0}(\bar{\beta})\cong \cap_{i<j} \langle x_i-x_j,y_i-y_j\rangle^{d_{ij}}.$$
By equivariant formality of $H_*(\Sp_\gamma)$, we then have in analogy to the equivalued case that 
$$\prod_{i<j}(y_i-y_j)^{d_{ij}}H^T_*(\Sp_\gamma)\cong  \bigcap_{i<j} \langle x_i-x_j,y_i-y_j\rangle^{d_{ij}} \otimes_{\C[\bx]}\C[\bx^\pm]\cong \HY^{a=0}(\bar{\beta})\otimes_{\C[\bx]}\C[\bx^\pm].$$
\end{remark}

\begin{remark}
It is not clear to us what the correct analogues, if any, of these 
link-theoretic notions are for other root data. While the definition of the HOMFLY homology as 
Hochschild homology of certain complexes of Soergel bimodules 
\cite{Khovanov} certainly makes sense in all types, many aspects of the 
theory, including the $y$-ification process, are undeveloped at the 
time. Work in progress by Hogancamp and Makisumi addresses some of these questions.
 
It is also an interesting question whether the resulting (Hochschild) homology of the (complex corresponding to the) full twist is parity, or related to $J_G$ for other types.
\end{remark}

\subsection{Hilbert schemes of points on curves}
It is useful to think of the link $\bar{\beta}$ from the previous section as the link of the plane curve singularity which is the pullback along $\gamma$ of the universal spectral curve over $\ft^{reg}/\Sn$. Recall that the {\em link} of $C\subset \C^2$ at $p\in C$ is the intersection of $C$ with a small three-sphere centered at $p$. In particular, $Link(C,p)$ is a compact one-manifold inside $S^3$, i.e. a link in the previous sense. Motivated by conjectures of Gorsky-Oblomkov-Rasmussen-Shende \cite{ORS, GORS}
there should then be a relationship of the affine Springer fibers, Hilbert schemes of points on the plane and link homology to the Hilbert schemes of the plane 
curve singularities $\{x^n=y^{dn}\}$. Namely, for $G=GL_n$ and 
$$\gamma=\begin{pmatrix}
a_1t^d&&\\
&\ddots&\\
&&a_nt^d
\end{pmatrix}$$ the characteristic polynomial of $\gamma$ is 
$$P(x)=\prod_i (x-a_it^d).$$ We may assume that $a_i=\zeta^i$ for $\zeta$ a primitive $n$th root of unity, in which case $P(x)=x^n-t^{dn}$. This determines a spectral curve in $\A^2$ with coordinates $(x,t)$, with a unique singularity at zero. It has a unique projective model with rational components and no other singularities. Call this curve $C$.

The compactified Jacobian of any curve $C$, denoted $\bJac(C)$, is by definition the moduli space of torsion-free rank one, degree zero sheaves 
on $C$. It is known by eg. \cite{Ngo} that in the case when $C$ has at worst planar singularities (and is reduced), we have a homeomorphism of stacks 
\begin{equation}\label{eq:producttheorem}\bJac(C)\cong \Jac(C)\times^{\prod_{x\in C^{sing}\Jac(C_x)}}\prod_{x\in C^{sing}}\bJac(C_x),\end{equation}
where $\bJac(C_x)$ is a local version of the compactified Jacobian at a closed point $x\in C$, sometimes also called the Jacobi factor.
In the case when $C=\{x^n=t^{dn}\}$, we have just a unique singularity 
and rational components, so that Eq. \eqref{eq:producttheorem} becomes 
a homeomorphism between a quotient of the moduli of fractional ideals in $\text{Frac}(\C[[x,y]]/x^n-y^{dn})$ and the compactified Jacobian. From the lattice 
description of the affine Grassmannian, it is not too hard to show that 
this former space actually equals $\Sp_\gamma/\Lambda$ \cite{MY}.

It is an interesting problem to determine the Hilbert schemes of points $C^{[n]}$ on these curves. These are naturally related to the 
compactified Jacobians via an Abel-Jacobi map, which has a local 
version as well. In the case when $C$ is integral, it is known that 
the global map becomes a $\P^{n-2g}$-bundle for $g\gg 0$, and respectively an isomorphism in the local case. In general we only 
know that it is so for a union of irreducible components of the compactified Picard, of which 
there are infinitely many (for each connected component) in the case when $C$ has locally reducible singularities.

In \cite{Kiv}, we have initiated an approach to computing $H_*(C^{[n]})$ where $C$ is reducible, using a certain algebra action on 
$$V:=\bigoplus_{n\geq 0} H_*(C^{[n]}).$$ Note that this is a bigraded vector space, where one of the gradings is given by the number of points $(n,0)$, and the other one is given by the homological degree $(0,j)$.
\begin{theorem}[\cite{Kiv}]
Let $$A_m:=\C[x_1,\ldots,x_m,\partial_{y_1},\ldots,\partial_{y_m},\sum_i \partial_{x_i}, \sum_i y_m]\subset \text{Weyl}(\A^{2m}),$$
where $x_i$ carries the bigrading $(1,0)$ and $y_i$ the bigrading $(1,2)$. Suppose $C$ is locally planar and has $m$ irreducible components. Then there is a geometrically defined action 
$A_m\times V\to V.$ 
\end{theorem}

Roughly speaking, the action on $V$ is given as follows. For a fixed component $C_i$ of $C$, the operator $x_i: V\to V$ adds points, and $\partial_{y_i}$ removes them. These are defined using a choice of a point $c_i\in C_i$ and a corresponding embedding $C^{[n]}\hookrightarrow C^{[n+1]}.$ On the other hand, the operator $\sum_{i}\partial_{x_i}: V\to V$ removes a "floating" point and $\sum_i y_i$ adds a floating point. These are defined as Nakajima correspondences.

The original computation of $T$-equivariant homology of affine Springer fibers in \cite{GKM2} for $G=GL_2$ bears a striking resemblance to the second main result in \cite{Kiv}. 
In particular, if $C$ is the union of two projective lines along a point, $$V\cong \frac{\C[x_1,x_2,y_1,y_2]}{(x_1-x_2)\C[x_1,x_2,y_1+y_2]}.$$ Furthermore, when $G=GL_2$, we have 
$$H_{*,ord}^T(\Sp_{tz})=\frac{\C[x_1^\pm,x_2^\pm,y_1,y_2]}{(x_1-x_2)\C[x_1^\pm,x_2^\pm,y_1+y_2]}.$$ Here $H_{*,ord}^T(-)$ means the Borel construction of ordinary $T$-equivariant homology. See Theorem \ref{thm:gkmmain}  for a more general statement.

Based on computations in \cite{Kiv} and some new examples in Section \ref{sec:cptfdjacobians}, we are 
lead to conjecture the following.  
\begin{conjecture}
Let $C$ be the (unique) compactification with rational components and no other singularities of the curve $\{x^n=y^{dn}\}$. 
Then as a bigraded $A_n$-module, we have 
\begin{equation}
V:=\bigoplus_{m\geq 0} H_*(C^{[m]},\Q)\cong 
\frac{\C[x_1,\ldots, x_n,y_1,\ldots, y_n]}{\sum_{i\neq j}\sum_{k=1}^{d} 
(x_i-x_j)^k \ker(\partial_{y_i}-
\partial_{y_j})^k}.
\end{equation}
\end{conjecture}

\subsection{Organization}

The organization of the paper is as follows. In Section 
\ref{sec:affinespringer} we give background on affine Springer 
fibers. In Section \ref{sec:BMhomology} we compute the 
torus equivariant Borel-Moore homology of the affine Springer fibers we are 
interested in, following Goresky-Kottwitz-MacPherson and Brion. In Section 
\ref{sec:isospectral}, we give background on Hilbert schemes of points 
on the plane and relate results from the previous sections with those of Haiman. We also discuss our results and their implications in this direction for arbitrary $G$ in Section \ref{sec:isospectral-othertypes}. In Section \ref{sec:braids}, we relate the equivariant Borel-Moore homology of affine Springer fibers with braid theory, and in the type A case with the knot homology theories of Khovanov-Rozansky and Gorsky-Hogancamp.
Finally, in 
Section \ref{sec:cptfdjacobians} we compute some new examples and make a conjecture describing the structure 
of the homology of Hilbert schemes of points on the plane curves $\overline{\{x^n=y^{dn}\}}$.

\subsection*{Acknowledgements}
I would like to thank Erik Carlsson, Victor Ginzburg, Eugene Gorsky, Mark Haiman,
Matthew Hogancamp, Alexei Oblomkov, Hiraku Nakajima, Peng Shan, Eric Vasserot and 
Zhiwei Yun for interesting discussions on 
the subject. Special thanks to Roman Bezrukavnikov for suggesting the 
reference \cite{GKM2}. Additionally, I would like to thank the MSRI for 
hospitality, for most of the work was completed there during Spring 
2018.
This research was supported by the NSF grants DMS-1700814  and 
DMS-1559338, as well as the Ville, Kalle and Yrj\"o V\"ais\"al\"a 
foundation of the Finnish Academy of Science and Letters. 
\section{Affine Springer fibers}
\label{sec:affinespringer}
\label{sec:affinespringer-defs}
In this section, we define the affine Springer fibers we are 
considering. For more details on the definitions, see the notes of Yun 
\cite{Yun}.
Let $G$ be a connected reductive group over $\C$. Choose $T\subset B
\subset G$ a maximal torus and a Borel subgroup as per usual. We denote 
the Lie algebras of $G,B,T$ respectively by $\fg,\fb,\ft$. 

 Denote the lattice of cocharacters $X_*(T)=\Lambda$ and the Weyl group 
$W$. Let the extended affine Weyl group be $\widetilde{W}:=\Lambda\rtimes W$. We use this convention to align with \cite{GKM2}. 

If $R$ is a $\C$-algebra and $F$ represents an fpqc sheaf out of $
\text{Aff}/\C$, we let $F(R)$ be the associated functor of points 
evaluated at $R$ (for an excellent introduction to these notions in the 
context we are interested in, see notes of Zhu \cite{Zhu}). Often when $R=\C$, we omit 
it from the notation and simply refer by $F$ to the closed points.

Denote the affine Grassmannian of $G$ by $\Gr_G$ and its affine flag 
variety by $\Fl_G$. These are naturally ind-schemes. If $G=GL_n$, we 
will often write just $\Gr_n$ and $\Fl_n$. Write $\cK=\C((t))$ and $
\cO=\C[[t]]$. Then $\Gr_G(\C)=G(\cK)/G(\cO)$ and $\Fl(\C)=G(\cK)/\bI$, 
where $\bI$ is the Iwahori subgroup corresponding to the choice of $B$ 
and the uniformizer $t$. Let $\tT:=T\rtimes \G_m^{rot}$ be the extended torus, where $a\in \G_m^{rot}$ scales $t$ by $t\mapsto at$.

There is a left action of $T(\C)$ on $\Gr_G(\C)$ and $\Fl_G(\C)=G(\cK)/
\textbf{I}$. This action is topological in the analytic topology. Its 
fixed points are determined using the following Bruhat decompositions: 
$$G(\cK)=\bigsqcup_{\lambda\in \Lambda} \bI t^\lambda G(\cO)=
\bigsqcup_{w\in \tW} \bI t^w \bI.$$
Since $T(\C)$ acts nontrivially on the real affine root spaces in $\bI$, and fixes the cosets $t^\lambda G(\cO)$, $t^w\bI$ respectively, we see that the 
fixed point sets are discrete, and in a natural bijection with $
\Lambda, \tW$. 

\begin{definition}
Let $\gamma\in \Lie(G)\otimes_\C \cK$. The affine Springer fibers $\Sp_
\gamma\subset \Gr_G$ and $\tSp_\gamma\subset \Fl_G$ are defined as the 
reduced sub-ind-schemes of $\Gr_G$ and $\Fl_G$ whose complex points are 
given by 
$$\Sp_\gamma(\C)=\{gG(\cK)|g^{-1}\gamma g\in \Lie(G)\otimes_\C \cO\}$$
$$\tSp_\gamma(\C)=\{g\textbf{I}|g^{-1}\gamma g\in \Lie(\textbf{I})\}.$$
\end{definition}

\section{Equivariant Borel-Moore homology of affine Springer fibers}
\label{sec:BMhomology}
In this section, we prove the main theorem of this paper, Theorem \ref{thm:mainthm}. We thank Eric Vasserot and Peng Shan for pointing out a mistake in the previous formulation and proof of Lemma \ref{lem:sl2}. They have informed us that they have independently found a similar solution to the issue.

\subsection{Borel-Moore homology}
We now review equivariant Borel-Moore homology. The paper \cite{Bri} is the main reference for this section.
For a projective (but not necessarily irreducible) variety $X$, one defines the Borel-Moore homology as 
$H_*(X):=H^{-*}(X,\omega_X),$ where $\omega_X$ is the Verdier dualizing complex in $D^b_c(X)$. Note that we use $H_*(-)$ for {\em Borel-Moore} homology, not the usual singular or {\'e}tale homologies.

For a $T$-variety $X$, where $T\cong \G_m^n$ is a diagonalizable torus, imitating the Borel construction of equivariant (co)homology is not completely straightforward, as the classifying space $BT$ is not a scheme-theoretic object. However, using approximation by $m$-skeleta as in \cite{Bri}, or a simplicial resolution of $BT$ as in \cite{BeLu}, one gets around the issue by defining $$H_k^T(X):=H_{k+2mn}(X\times^T ET_{m}), \; m\geq \dim X-k/2.$$
Here $ET_m:=(\C^{m+1}-0)^d$ with the $T$-action $(t_1,\ldots, t_d)\cdot(v_1,\ldots,v_d)=(t_1v_1,\ldots,t_dv_d)$. This action is free, and the quotient $ET_m\to (\P^m)^d$ is a principal $T$-bundle.

The above definition of $H_k^T(X)$ is independent of $m$ as follows from the Gysin isomorphism $H_{k+2m'n}(X\times^T ET_{m'})\to H_{k+2mn}(X\times^TET_{m})$ for $m'\geq m\geq \dim X-k/2$. Note that $H_*^T(X)$ is a graded module over $H^*_T(X)$ via the cap product and in particular a graded module over $H_*^T(pt)$.

Recall that $X$ is {\em equivariantly formal} (see \cite{GKM1, GKM2}) if the Leray spectral sequence 
$$H^p(BT,H^q(X))\Rightarrow H_T^{p+k}(X)$$ degenerates at $E_2$.
If $X$ is equivariantly formal, then $H_*^T(X)$ is a {\em free} $H_T^*(pt)$-module \cite[Lemma 2]{Bri}.

The above definition of $H_*^T(-)$ enjoys some of the usual localization properties, as studied e.g. in \cite{Bri}.
For example, we have an ''Atiyah-Bott" formula \cite[Lemma 1]{Bri}. 
\begin{theorem}
Suppose the $T$-action on $X$ has finitely many fixed points. Let $i_*: H_*^T(X^T)\to H_*(X)$ be the $\C[\ft]$-linear map given by the inclusion of the fixed-point set to $X$. Then $i_*$ becomes an isomorphism after inverting finitely many characters of $T$.
\end{theorem}
From the perspective of commutative algebra, it is useful to note the following from \cite[Proposition 3]{Bri}.
\begin{proposition}
\label{prop:dualizingmodule}
If $X$ is equivariantly formal, then $$H^T_*(X)\cong \Hom_{\C[\ft]}(H^*_T(X),\C[\ft]).$$
The map is given by $$\alpha \mapsto (\beta\mapsto p_{X*}(\beta\cap \alpha)),$$ where $p_X: X\times^T ET\to BT$ is the projection.
\end{proposition}

Another localization theorem was proved in \cite[Theorem 7.2]{GKM1} for $T$-equivariant (co)homology. As in \cite[Corollary 1]{Bri}, it is translated to Borel-Moore homology as follows.

\begin{proposition}
\label{localization}
Let $X$ be an equivariantly formal $T$-variety containing only finitely many orbits of dimension $\leq 1$. Then $H^T_*(X)\cong i_*^{-1} H_*^T(X) \subset H_*^T(X^T)\otimes \C(\ft)$ consists of all tuples $(\omega_x)_{x\in X^T}$ of rational differential forms on $\ft$ satisfying the following conditions. 
\begin{enumerate}
\item The poles of each $\omega_x$ are contained in the union of singular hyperplanes and have order at most one. Recall that a {\em singular hyperplane} in $\ft$ is the vanishing set of $d\chi$, where $X^{\ker\chi}\neq X^T$ and $\ker\chi$ is the codimension one subtorus of $T$ defined by $\chi$.
\item For any singular character $\chi$ and for any connected component $Y$ of $X^{\ker \chi}$, we have $$\Res_{\chi=0}\left(\sum_{x\in Y^T}\omega_x\right)=0.$$
\end{enumerate}
\end{proposition}
As the number of orbits of dimension $\leq 1$ is finite, and the closure of each one-dimensional orbit contains exactly two fixed points (see \cite{GKM1}), it is natural to form the graph whose vertices are the fixed points and edges correspond to one-dimensional orbits. We call the associated weighted graph whose edges are labeled by the differentials $d\chi$ of singular characters the {\em GKM graph}. 

Note that it is easy to recover $H_*(X)$ from $H_*^T(X)$ for equivariantly formal varieties by freeness, as shown in \cite[Proposition 1]{Bri}. Namely, we have
\begin{proposition}
Let $T'\subset T$ be a subtorus. Then 
$$H_*^{T'}(X)\cong \frac{H^T_*(X)}{\text{Ann}(\ft')\cdot H^*},$$
where $\text{Ann}(\ft')\subset \C[\ft]$ is the annihilator of $\ft'=\Lie(T')$.
In particular, when $T'$ is trivial, we get
$$H_*(X)=\frac{H_*^T(X)}{\C[\ft]_+ H_*^T(X)}.$$
\end{proposition}

Ultimately, we are interested in the equivariant Borel-Moore homology of the ind-projective varieties $\Sp_{t^dz}$. Suppose now that $X=\varinjlim X_i$ is an ind-scheme over $\C$ given by a diagram $$X_0\subset X_1\subset X_2\subset \cdots$$ where the maps are $T$-equivariant closed immersions and each $X_i$ is projective. By properness and the definition of $H_*^T(-)$, there are natural pushforwards  
$$H_*^T(X_i)\to H_*^T(X_{i+1}),$$ using which we
{\em define} $$H_*^T(X):=\varinjlim H_*^T(X_i).$$ The usual (non-equivariant) Borel-Moore homology is defined similarly. Note that since the $X_i$ are varieties we are still abusing notation and mean $X_i(\C)$ when taking homology.

\begin{remark}
While $H_*(-)$ and $H_*^T(-)$ could be defined for any finite-dimensional locally compact, locally contractible and $\sigma$-compact topological space $X$ using the sheaf-theoretic definition \cite[Corollary V.12.21.]{Bre}, it is {\em not} true that this definition gives the same answer for $X(\C)$ as the above definition (there's always a map in one direction). For example, if $X(\C)=\varinjlim [-m,m]\cong \Z$ is the colimit of the discrete spaces $[-m,m]\subset \Z$, which are of course also the $\C$-points of a disjoint union of $2m+1$ copies of $\A^0$, then $H^{-*}(X,\omega_X)\cong \C^{\Z}$ is the homology of the one-point compactification of $\Z$ with the cofinite topology, while treating $X$ as an ind-variety we get $H_*(X)\cong \C^{\oplus \Z}$. 
\end{remark}

Call a $T$-ind-scheme $X$ {\em equivariantly formal} if each $X_i$ is equivariantly formal and $T$-stable. Call it {\em GKM} if each $X_i$ has finitely many orbits of dimension $\leq 1$. We have the following corollary to Theorem \ref{localization}.

\begin{corollary}
\label{cor:localization}
Let $X$ be an equivariantly formal GKM $T$-ind-scheme. Then 
$H_*^T(X)\subset H_*^T(X^T)\otimes \C(\ft)$ consists of all tuples 
$(\omega_x)_{x\in X^T}$ of rational differential forms on $\ft$ satisfying the conditions in Theorem \ref{localization}.
\end{corollary}
\begin{proof}
By assumption, we have inclusions of $T$-fixed points 
$X_i^T\to X_{i+1}^T$ and their union is $X^T$. 
Taking the colimit of $H_*^T(X_i)\hookrightarrow H_*^T(X_i^T)$, we get by exactness
$$\iota: H_*^T(X):=\varinjlim H_*^T(X_i)\hookrightarrow  \varinjlim H_*^T(X_i^T)=:H_*^T(X^T),$$
which becomes an isomorphism when tensoring with $\C(\ft)$. Any tuple $(\omega_x)_{x\in X^T}$ of rational differential forms (of top degree) on $\ft$ inside $\iota^{-1}_*H_*^T(X)$ has some $i$ such that it is in the image of $\iota^{-1}_* H^T(X_i)$. By Proposition \ref{localization}, it therefore satisfies the desired conditions.
\end{proof}
\begin{remark}
While the number of fixed points and one-dimensional orbits might now be infinite, we may still form the (possibly infinite) GKM graph.
\end{remark}

\subsection{The $SL_2$ case}
\label{sec:sl2}
We first prove Theorem \ref{thm:mainthm} in the case $G=SL_2$. Recall that $\tT=T(\C)\times \C^*\subset G((t))$ denotes the extended torus. As shown in \cite[Lemma 6.4]{GKM2}, for $G=SL_2$ the one-dimensional $\tT$-orbits of $X_d:=\Sp_{t^dz}$ are given as follows. If we identify $\Sp_{t^dz}^{\tT}=\Z$, then there is an orbit between $a,b\in \Z$ if and only if $|a-b|\leq d$. Moreover, $\tT$ acts on this orbit through the character (in fact, real affine root) $(\alpha, a+b)\in X_*(\tT)\cong \Lambda\times \Z$. Identify further the differential of this character by $y+(a+b)t\in \C[\widetilde{\ft}]$.

Recall that the affine Grassmannian of $SL_2$ decomposes as the the disjoint union of finite-dimensional Schubert cells $\Gr_{SL_2}^m:=SL_2(\cO)t^\lambda SL_2(\cO)$. Let $\Gr^{\leq m}_{SL_2}=\overline{\Gr^m_{SL_2}}=\bigsqcup_{l \leq m} \Gr_{SL_2}^l$.  It is clear that the subvarieties $X_d^{\leq m}:=(\Sp_{t^dz})^{\leq m}=\overline{\Sp_{t^dz}\cap \Gr_{SL_2}^{\leq m}}$ are $\tT$-stable. The corresponding GKM graph is just the induced subgraph formed by the vertices $[-m,m]\subset \Z$. In particular, we may compute $H_*^{\tT}(X_m)$ using Theorem \ref{localization} for the corresponding GKM graphs. Note that each such graph in this case is a chain of complete graphs on $d$ vertices glued along $d-1$
vertices. Let us first practice the case when the length of the chain is one, i.e. we are computing the $\tT$-equivariant Borel-Moore homology of the classical Springer fiber $sp_e\subset \Gr(2d,d)$, where $e$ is the square 
of a regular nilpotent element (see \cite{ChLa}).  This is essentially a projective space of dimension $d$.
\begin{example}
\label{ex:p1}
Let $d=1$. Then the GKM graph of $sp_e$ is two vertices joined by a line, with the character $y+t$.
Theorem \ref{localization} then tells us that 
$$i_*: H^T_*(sp_e^T)\to H^T_*(sp_e)$$ is injective and $(i_*)^{-1}H_*^T(sp_e)$ consists of rational differential forms $(\omega_0,\omega_1)$ so that 
$$\Res_{y=-t}(\omega_0+\omega_1)=0$$ with poles of order at most one and along $y=-t$. In particular, any polynomial linear combination of $a=(\frac{dydt}{y+t}, \frac{-dydt}{y+t})$ and $b=(dydt,0)$ satisfies these requirements and is the most general choice, so we conclude $H^T_*(X)$ is a free $\C[y,t]$-module with basis $a,b$. As $sp_e=\P^1$ is smooth, we further use the Atiyah-Bott localization theorem to conclude that $a=[\P^1]$.
\end{example}
From now on, we will save notation and write each tuple of differential forms $(\omega_1,\ldots, \omega_q)=(f_1dydt,\ldots,f_qdydt)$ simply as $(f_1,\ldots f_q)$.

Let us now compute $H_*^T(\Sp_{tz})$ for $G=SL_2$ for illustrative purposes. This is very similar to Example \ref{ex:p1}.
\begin{proposition}
\label{prop:sl2d1}
If $d=1$ and $G=SL_2$, then $H_*^{\tT}(\Sp_{tz})$ is the $\C[t,y]$-linear span of $$
a=(\ldots,0,0,1,0,0,\ldots)$$ and 
$$
b_i=(\ldots, 0,\frac{1}{(2i+1)t+y}, \frac{-1}{(2i+1)t+y},0,\ldots),$$
where the $1$ in $a$ is at the $0$th position and the nonzero entries in $b_i$ are at the $i$th and $(i+1)$th positions, respectively.
In particular, $$\frac{H_*^{\tT}(X)}{t\cdot H_*^{\tT}(X)}\cong H_*^T(X)$$ is isomorphic to the $\C[y]$-linear span of $a$ and $b_i'=(\ldots,0,1/y,-1/y,0,\ldots).$
\end{proposition}
\begin{proof}
By the discussion above, the GKM graph has vertices $\Z$ and edges exactly between $i, i+1$ for all $i$. Indeed, it is well-known that $X_1$ is just an infinite chain of projective lines. The weights of the edges for the $\tT$-action are given by the character $(2i+1)t+y$ by \cite[Lemma 6.4.]{GKM2}. Applying Corollary \ref{cor:localization} we get the first claim. Setting $t$ to zero recovers $H_*^T(X)$, so that we get the second result.
\end{proof}

\begin{lemma}
\label{lem:sl2}
Let $d\geq 1$. Then the $\tT$-equivariant Borel-Moore homology of $X_d=\Sp_{t^dz}$ is the $\C[t,y]$-linear span of 

\begin{align*}
a_0&=(\ldots,0,0,1,0,0,\ldots)\\
a_1&=(\ldots,0,0,\frac{1}{y+t},\frac{-1}{y+t},0,\ldots)\\
& \vdots \\
a_{d-1}&=(\ldots,0,0,\frac{1}{\prod_{i=1}^{d-1}(y+it)},\frac{-\binom{d-1}{1}}{(y+t)\prod_{i=2}^{d-1}(y+(i+1)t)},\ldots, \frac{(-1)^{d-1}\binom{d-1}{d-1}}{(\prod_{i=1}^{d-1}(y+(d-1+i)t)},0,\ldots)\\
b_k&=(\ldots,0,0,\frac{\binom{d}{0}}{f_k^{(1)}},\frac{-\binom{d}{1}}{f_k^{(2)}},\ldots, \frac{(-1)^d\binom{d}{d}}{f_k^{(d)}},0,\ldots), \; k\in \Z,
\end{align*}
where 
$$f_k^{(j)}=\prod_{i=0}^{j-1}(y+(2k+i+j)t)\prod_{i=j+1}^d(y+(2k+i+j)t),\;\; j=1,\ldots,d.$$
Here the nonzero entries in $a_i$ are at $0,\ldots, i$ and the nonzero entries in $b_k$ are at $k,\ldots, k+d$.

In particular, letting $t=0$, $$H^T_*(X_d)\subseteq H^T_*(\Lambda)$$ is the $\C[y]$-linear span of 
\begin{align*}
a_0'&=(\ldots,0,0,1,0,0,\ldots)\\
a_1'&=(\ldots,0,0,\frac{1}{y},\frac{-1}{y},0,\ldots)\\
& \vdots \\
b_k'&=(\ldots,0,0,\frac{\binom{d}{0}}{y^d},\frac{-\binom{d}{1}}{y^d},\ldots, \frac{(-1)^{d-1}\binom{d}{d-1}}{y^d}\frac{(-1)^d\binom{d}{d}}{y^d},0,\ldots), \; k\in \Z.
\end{align*}
Note that if we write $\C[\Lambda]=\C[x^\pm]$, then in the monomial basis $a_0'=x^0$, $a_1'=\frac{1-x}{y}$, and $b_k'=x^k(1-x)^d/y^d$.
\end{lemma}
\begin{proof}
Let us first check the residue conditions of Corollary \ref{cor:localization}. Note that $a_0,\ldots, a_{d-1}$ are just $b_0$ for some smaller $d$, in particular it is enough to check the conditions for $b_k$. There is an orbit between $k+j$ and $k+j'$ whenever $|j-j'|\leq d$, and $\tT$ acts on said orbit via $\chi=y+(2k+j+j')t$. In particular, we need to prove that 
$$\Res_{y=-(2k+j+j')t}\left(\frac{(-1)^j\binom{d}{j}}{f_k^{(j)}}+\frac{(-1)^{j'}\binom{d}{j'}}{f_k^{(j')}}\right)=0.$$
First, we compute that
$$f_k^{(j)}=\prod_{i\neq j,1\le i\le d}(y+(2k+i+j)t),$$ so the residue at $y=-(2k+j+j')t$ of 
$1/f_k^{(j)}$ equals
$$\frac{1}{\prod_{i\neq j,j'}(i-j')t}=\frac{(j-j')}{\prod_{i\neq j'}(i-j')t}=\frac{(j-j')}{(-1)^{j'} (j')! (d-j')!}.$$
If we multiply this by $$(-1)^j \binom{d}{j},$$ we get 
$$\frac{(j-j')d!}{(-1)^{j'+j} (j')! (d-j')!j!(d-j)!},$$
which is antisymmetric under switching $j$ and $j'$. By linearity of taking residues, we get the result.

We need to show the reverse inclusion. Let $sp_d$ be the Spaltenstein variety of $d$-planes in $\C^{2d}$ 
stable under the $(d,d)$-nilpotent element. From \cite[page 448]{ChLa}, we know that $X_d$ is an 
infinite chain of $sp_d$ glued along $sp_{d-1}$. In addition, $X_d^{\leq m}$ from the beginning 
of Section \ref{sec:sl2} is a chain of $2m$ copies of $sp_d$ glued along $sp_{d-1}$. From the 
form of the GKM graph it is immediate that the $T$-equivariant Borel-Moore homology of $X_{d}^{\leq m}$
 as a graded $\C[y,t]$-module looks like that of a chain of $2m$ copies of $\P^d$ consecutively glued along $\P^{d-1}$. In particular, 
$H_*^T(X_d^{\leq m})$ has rank $1$ over $\C[y,t]$ in degrees $\leq 2d-2$ and rank $2m$ in degree $2d$. Since the classes $b_i$ for $i=-m,\ldots, m$ are linearly independent over $\C[y,t]$ and there are $2m$ of them, the $b_i$ must span $H_{2d}^T(X_d^{\leq m})$. Taking the colimit, the first result follows. The second result is immediate from the form of $f_k^{(j)}$ and setting $t=0$.

\end{proof}

\begin{remark}
In \cite[Section 12]{GKM2}, the analogues of the classes $b_k$ are played by the polynomials denoted $f_{k,d}$ in {\it loc. cit.} They are the ones attached to ''constellations" of one-dimensional orbits.
\end{remark}
\begin{remark}
\label{rmk:sl2schubert}
In Proposition \ref{prop:sl2d1} and Lemma \ref{lem:sl2}, the polynomials $f_k^{(j)}$ that appear seem to be related to the affine Schubert classes in $H_*^T(X_d)$ given by intersections by $G(\cO)$-orbits on $\Gr_{SL_2}$. Since the components $\cong sp_d$ are rationally smooth (by e.g. the criteria in \cite[Theorem 1.4]{Brion2}), $f_k^{(j)}$ are exactly the inverses to $\tT$-equivariant Euler classes of the $k$th irreducible component at the fixed point $j\in \Lambda$. It seems that for higher rank groups, rational smoothness of the irreducible components no longer holds in general.
\end{remark}
\subsection{The general case}
In this section, we prove Theorem \ref{thm:mainthm}.
The GKM graph for $\tT$ acting on $\Sp_{t^dz}$ is always infinite; indeed we have the following. 
\begin{lemma}
The vertices of the GKM graph of $\Sp_{t^dz}$ are $\Lambda=X_*(T)$ and there is an edge $\lambda\to \mu$ whenever $\lambda-\mu=k\alpha$, where $\alpha\in \Phi^+$ and $k\leq d$.
\end{lemma} 
\begin{proof}
From \cite[Lemma 5.12]{GKM2}, we know that the one-dimensional $\tT$-orbits are $(\Sp_{t^dz})_1=\bigcup_{\alpha\in \Phi^+} (\Sp_{t^dz}^\alpha)_1$ and $\Sp_{t^dz}^\alpha\cap \Sp_{t^dz}^\beta=\Lambda$ unless $\beta=\alpha$.
In particular, we are reduced to the semisimple rank 1 case which is reduced to the $SL_2$ case by \cite[Lemma 8.1]{GKM2} and the $SL_2$ case is handled by Lemma 6.4 in {\em loc. cit.}. 
\end{proof}

We also need the following corollary to Lemma \ref{lem:sl2}.
\begin{corollary}
\label{cor:rk1}
Let $\alpha\in \Phi^+$, and let $y_\alpha\in \C[\ft]=H^*_T(pt)$ be the linear functional corresponding to $\alpha$. Denote
$X_d^\alpha:=\Sp_{zt^d}^\alpha:=\Sp_{zt^d}\cap \Gr_{H^\alpha}$. For any $G$ and $\alpha\in \Phi^+(G,T)$, we have
$$y_\alpha^d H_*^T(X_d^\alpha)=J^d_\alpha=\langle y_\alpha, 1-\alpha^\vee\rangle^d \subset H_*^T(\Lambda)= \C[\Lambda]\otimes \C[\ft].$$ Here $\langle S \rangle$ means the ideal in $\C[\Lambda]\otimes \C[\ft]$ generated by the subset $S$.
\end{corollary}
\begin{proof}
Since $X_d^\alpha$ is an unramified affine Springer  fiber of valuation $d$ for a semisimple rank one group, it is a disjoint union of infinite chains of Spaltenstein varieties $sp_d$, as explained in Section \ref{sec:sl2}. More precisely,
it is a disjoint union of such over $\Lambda/\langle \alpha^\vee\rangle$ inside $X_d$. Identify $H_*^T(\Lambda)$ with $\C[\Lambda]\otimes \C[\ft]$ and write its elements $\C[\ft]$-linear combinations of $x^\lambda:=x^\lambda\otimes 1$.  From Lemma \ref{lem:sl2} and \cite[Lemma 6.4]{GKM2}, we have that 
$H^*_T(X_d^\alpha)\subset H_*(\Lambda)\otimes \C(\ft)$ is the $\C[\ft]$-linear span of 
$$\frac{x^\lambda(1-x^{\alpha^\vee})^d}{y^d_\alpha}$$ and 
$$\frac{(1-x^{\alpha^\vee})^k}{y^k_\alpha}$$ for $k=0,\ldots, d-1$. In particular, $y^d_\alpha H^*_T(\Sp_{zt}^\alpha)\subset \C[\Lambda]\otimes \C[\ft]$ is identified with the ideal $$J^d_\alpha=\langle (1-x^{\alpha^\vee})^d,(1-x^{\alpha^\vee})^{d-1}y_\alpha,\ldots, (1-x^{\alpha^\vee})y_\alpha^{d-1},y_\alpha^{d}\rangle.$$
\end{proof}

\begin{theorem}
\label{thm:mainthm}
Let $\Delta=\prod_\alpha y_\alpha \in H^*_T(pt)$ be the Vandermonde element. The equivariant Borel-Moore homology of $X_d:=\Sp_{t^dz}$ for a reductive group $G$ is up to multiplication by $\Delta^d$ canonically isomorphic as a (graded) $\C[\Lambda]\otimes\C[\ft]$-module to the ideal
$$J^{(d)}=\bigcap_{\alpha\in \Phi^+} J_\alpha^d \subset \C[\Lambda]\otimes \C[\ft].$$ In particular, there is a natural algebra structure on $\Delta^d H_*^T(\Sp_\gamma)$ inherited from $\C[\Lambda]\otimes \C[\ft]$, and $J^{(d)}$ is a free module over $\C[\ft]$.
\end{theorem}
\begin{proof}
By \cite[Lemma 5.12]{GKM2} and Corollary \ref{cor:localization}, we have that $H_*^T(X_d)=\bigcap_\alpha H_*^T(X_d^\alpha)\subset H_*^T(\Lambda)\otimes \C(\ft)$.
By equivariant formality and Corollary \ref{cor:localization}, we furthermore have that  $$\Delta^d\cdot H_*^T(X_d)\subset H^T_*(\Lambda)$$ is a free $\C[\ft]$-module.
Since $J_\alpha^d=y_\alpha^d H_*^T(\Sp_{t^dz}^\alpha)$ contains $\Delta$, we must have $\Delta^d\cdot H_*^T(X_d)\subseteq J_\alpha^d$ for all $\alpha$. Inverting $\Delta$, we see that $$\Delta^d\cdot H_*^T(X_d)_\Delta\cong \left(\bigcap_\alpha J_\alpha^d\right)_\Delta.$$ But $\Delta^d\cdot H_*^T(\Sp_{tz})$ was free over $\C[\ft]$, so by \cite[Lemma 6.14]{GoHo}, we have that $J^{(d)}=\Delta^d\cdot H_*^T(X_d)$.
\end{proof}

\begin{remark}
A priori, it is not at all obvious that $H_*^T(\Sp_{t^dz}^\alpha)$ would be a $\C[\Lambda]$-submodule of $H^T_*(\Lambda)$. The product structure on $H^T_*(\Lambda)$, while obvious in the algebraic statements, is geometrically a {\em convolution product}. In fact, it is the convolution product on the affine Grassmannian of $T$, as discussed in \cite{BFM}, and more recently \cite{BFN2} in the guise of a ''$3d$ $\cN=4$ Coulomb branch for $(T,0)$".
Moreover, it is also nontrivial that $y_\alpha^d H^T_*(\Sp_{t^dz}^\alpha)$ should have a natural subalgebra structure.
\end{remark}
\begin{remark}
It seems difficult to carry out analysis similar to Remark \ref{rmk:sl2schubert} for the case of general $G$. Erik Carlsson has informed us that he has performed computations related to $X_d$ using affine Schubert calculus (see also \cite{CO}). It would be interesting to relate the two approaches.
\end{remark}

\subsubsection{The affine flag variety}
In this section, we consider $Y_d=\tSp_\gamma,$ where $\gamma=zt^d$. We focus on the case $d=1$.
The $\tT$-fixed points of $Y_d$ are in a natural bijection with $\tW=\Lambda \rtimes W$. 
For $G=SL_2$, it is known that $Y_1$ is an infinite chain of projective lines again, and if we write elements of $\tW$ as $(k, w)$, $k\in \Z$, $w\in \{1, s\}$, there are one-dimensional orbits precisely between $(k,1)$ and $(k,s)$ as well as $(k+1,1)$ and $(k,s)$, see \cite[Section 13]{GKM2}. 
\begin{lemma}
When $G=SL_2$, we have that $H_*^{\tT}(Y_1)\subset H_*^{\tT}(\tW)$ is the $\C[y,t]$-linear span of the classes 
\begin{align*}
a_0&=(\ldots, 0,0,1,0,0,\ldots)\\
b_k&=(\ldots 0,0,\frac{1}{y+2kt},-\frac{1}{y+2kt},0,0,\ldots)\\
b_k'&=(\ldots 0,0,\frac{1}{y+(2k-1)t},0,0,-\frac{1}{y+(2k-1)t},0,0,\ldots)
\end{align*} 
where $b_k$ has nonzero entries at positions $(k,1)$ and $(k,s)$ and similarly $b_k'$ has nonzero entries at $(k,1)$ and $(k-1,s)$. In particular, by setting $t=0$, we get that 
$H_*^T(Y_1)$ is 
$$\left\lbrace \frac{1-s}{y}, \frac{1-x}{y}, 1 \right\rbrace \cdot \C[\Lambda]\otimes \C[\ft]\subset \C[\tW]\otimes \C[\ft].$$
\end{lemma}
\begin{proof}
The residue conditions needed to apply Corollary \ref{cor:localization} are almost exactly the same as in Proposition \ref{prop:sl2d1}. The second claim follows from the fact that in $\C[\tW]$, we may compute $$-(1-s)\cdot (\lambda,1)+(1-\alpha^\vee)\cdot (\lambda,1)=
-(\lambda,1)+(\lambda,s)+(\lambda,1)-(\lambda+1,1)=(\lambda,s)-(\lambda+1,1)= -b_k'|_{t=0}y.$$
\end{proof}
\begin{corollary}
\label{cor:affineflagsl2}
Let $y_\alpha\in \C[\ft]=H^*_T(pt)$ be the linear functional corresponding to $\alpha$ and 
$Y_d^\alpha:=\tSp_{zt^d}^\alpha:=\tSp_{zt^d}\cap \Fl_{H^\alpha}$. For any $G$ and $\alpha\in \Phi^+(G,T)$, we have
$$\tJ_\alpha:=y_\alpha H_*^T(Y_1^\alpha)=\left\lbrace 1-s_\alpha, 1-x^{\alpha^\vee}, y_\alpha \right\rbrace \cdot \C[\Lambda]\otimes \C[\ft]\subset \C[\tW]\otimes \C[\ft].$$ 
\end{corollary}
\begin{proof}
This is similar to Corollary \ref{cor:rk1} and \cite[page 547]{GKM2}. The affine Springer fiber $Y^\alpha_1$ is again a disjoint union of infinite chains of projective lines indexed by $\Lambda/\langle \alpha^\vee\rangle$. From this fact and the previous Corollary, we get that 
$H_*^T(Y^\alpha_1)$ is the $\C[\ft]$-linear span of $\frac{x^\lambda (1-x_\alpha)}{y_\alpha}, \frac{(1-s_\alpha)x^\lambda}{y_\alpha}$ and $1$. Multiplying by $y_\alpha$, we get the result.
\end{proof}
\begin{theorem}
\label{thm:affineflagmain}
For any reductive group $G$, 
$$\Delta \cdot H_*^{T}(Y_1)=\bigcap_\alpha \tJ_\alpha \subset \C[\tW]\otimes \C[\ft]$$
and furthermore $\tJ_G$ is a free module over $\C[\ft]$. Here $\Delta=\prod_\alpha y_\alpha$ as before.
\end{theorem}
\begin{proof}
The proof is entirely similar to Theorem \ref{thm:mainthm}.
\end{proof}
\begin{remark}
It is not at all clear from this description whether $\Delta \cdot H_*^{\tT}(Y_1)$ has an algebra structure. Based on Conjecture \ref{conj:affineflag} and the fact that there is a (noncommutative) algebra structure when $d=0$, it seems that this could be the case. 
\end{remark}

\subsubsection{Equivariant K-homology}
In this section, we state a version of Theorem \ref{thm:mainthm} in K-homology. We omit detailed proofs because they are entirely parallel to those in previous sections. 

In \cite{HHH}, more general equivariant cohomology theories, such as the equivariant K-theory of (reasonably nice) $T$-varieties is studied from the GKM perspective. Let $K^T(X)$ be the equivariant (topological) K-theory of a $T$-variety $X$. Following Proposition \ref{prop:dualizingmodule}, {\em define} the equivariant K-homology of $X$ as $$\Hom_{R(T)}(K^T(X),R(T)),$$ where $R(T)$ is the representation ring of $T$ over $\C$. In particular, fixing an isomorphism $T\cong \G_m^n$, we have $R(T)\cong \C[y_1^\pm,\ldots, y_n^\pm]$.

Adapting the description of \cite[Theorem 3.1]{HHH}, Proposition \ref{localization}, and Lemma \ref{cor:localization}, we have an analogue of Corollary \ref{cor:localization} in K-homology.

\begin{proposition}
Let $X$ be an equivariantly formal GKM $T$-ind-scheme. Then 
$K^T(X)\subset K^T(X^T)\otimes \C(\ft)$ consists of all tuples 
$(\omega_x)_{x\in X^T}$ of rational differential forms on $T$  satisfying the following conditions. 
\begin{enumerate}
\item The poles of each $\omega_x$ are contained in the union of singular divisors (i.e. of the form $\{y^\chi=1\}$ and have order at most one.
\item For any singular character $\chi$ and for any connected component $Y$ of $X^{\ker \chi}$, we have $$\Res_{y^\chi=1}\left(\sum_{x\in Y^T}\omega_x\right)=0.$$
\end{enumerate}
\end{proposition}

From this, it directly follows that we have the following complementary versions of Theorems \ref{thm:mainthm} and \ref{thm:affineflagmain}.
\begin{theorem}
\label{thm:mainthmktheory}
Let $\Delta'=\prod_{\alpha\in\Phi^+}(1-y^\alpha)\in R(T)$ be the Vandermonde element. The equivariant $K$-homology of $X_d:=\Sp_{t^dz}$ for a reductive group $G$ is up to multiplication by $(\delta')^d$ canonically isomorphic as a $\C[\Lambda]\otimes R(T)$-module to the ideal 
$$(J')^{(d)}:=\bigcap_{\alpha\in \Phi+} (J_\alpha')^d\subset \C[\Lambda]\otimes R(T).$$ Here $J'_\alpha:=\langle 1-y^\alpha, 1-x^{\alpha^\vee}\rangle$. The algebra structure on $(\Delta')^dH_*^T(\Sp_\gamma)$ is given by the convolution product on $K^T(\Lambda)$
\end{theorem}

\begin{theorem}
\label{thm:affineflagktheory}
For any reductive group $G$, 
$$\Delta \cdot K^{T}(Y_1)=\bigcap_\alpha \tJ_\alpha' \subset \C[\tW]\otimes R(T).$$ Here $$\tJ_\alpha'=\left\lbrace 1-x^{\alpha^\vee},1-y^\alpha,1-s_{\alpha}\right\rbrace\C[\Lambda]\otimes R(T)\subset \C[\tW]\otimes R(T).$$
\end{theorem}
\section{The isospectral Hilbert scheme}
\label{sec:isospectral}
\subsection{Definitions}
In this section, we define the relevant Hilbert schemes of points and list some of their properties. We then discuss the relationship of the results in Section \ref{sec:affinespringer} to the Hilbert scheme of points and the isospectral Hilbert scheme.

\begin{definition}
The Hilbert scheme of points on the complex plane, denoted $\Hilb^n(\C^2)$, is defined as the moduli space of length $n$ subschemes of $\C^2$. Its closed points are given by $$\{I\subset \C[x,y]|\dim_\C \C[x,y]/I=n\},$$ where $I$ is an ideal.
\end{definition}
\begin{definition}
The isospectral Hilbert scheme $X_n$ is defined as the following  reduced fiber product:
$$\begin{tikzcd}
X_n \arrow[d, "\rho"] \arrow[r]& \C^{2n}\arrow[d,"\cdot/\Sn"]\\
\Hilb^n(\C^2)\arrow[r,"\sigma"] &\C^{2n}/\Sn
\end{tikzcd}$$
\end{definition}
We have the following localized versions of these statements.
\begin{definition}\label{def:localized hilb definition}
The Hilbert scheme of points on $\C^*\times \C$ is the moduli space of length $n$ subschemes of $\C^*\times \C$. 
\end{definition}
Note that $\C^*\times \C$ is affine, so that the closed points of $\Hilb^n(\C^*\times \C)$ are given by $\{I\subset \C[x^\pm,y]|\dim_\C \C[x^\pm,y]/I=n, I \text{ ideal}\}$. In fact, $\Hilb^n(\C^*\times\C)$ is naturally identified with the preimage $\pi^{-1}((\C^*\times \C)^n/\Sn)$ under the Hilbert-Chow map 
$$\Hilb^n(\C^2)\to \C^{2n}/\Sn.$$
\begin{definition}\label{def: localized isospectral definition}
The isospectral Hilbert scheme on $\C^*\times \C$ is denoted $Y_n$, and defined to be the following reduced fiber product:
$$\begin{tikzcd}
Y_n \arrow[d, "\rho"] \arrow[r]& (\C^*\times \C)^n\arrow[d,"\cdot/\Sn"]\\
\Hilb^n(\C^*\times \C)\arrow[r,"\sigma"] &(\C^*\times \C)^n/\Sn
\end{tikzcd}$$
\end{definition}

Let $A=\C[\bx,\by]^{sgn}$ be the space of alternating polynomials. This is to be interpreted in two sets of variables, ie. taking the $sgn$-isotypic part for the diagonal action. We recall the following theorem of Haiman.

\begin{theorem}[\cite{Hai1}]
Consider the ideal $I\subset \C[\bx,\by]$ generated by $A$. Then for all $d\geq 0$, \begin{equation}\label{eq:anti-invts equal diagonals}
I^d=J^{(d)}=\bigcap_{i\neq j}\langle x_i-x_j, y_i-y_j\rangle^d\subseteq 
\C[\textbf{x},\textbf{y}].
\end{equation}
Moreover, $I^d$ is a free $\C[\by]$-module, and by symmetry, a free $\C[\bx]$-module. 
\end{theorem}
\begin{remark}
$J^{(d)}$ is not free over $\C[\bx,\by]$.
\end{remark}
We have the following corollary to Theorem \ref{thm:mainthm}, as stated earlier.
\begin{corollary}
The ideal $J^{(d)}\subset \C[\bx,\by]$ is free over $\C[\by]$.
\end{corollary}

The ideals $I^d=J^d=J^{(d)}$ and the space of alternating polynomials naturally emerge in the study of Hilbert schemes of points on the plane. 
\begin{theorem}\label{thm: blow-up description hilb}
The schemes $\Hilb^n(\C^2)$ and $X_n$ admit the following descriptions:
\begin{equation}
\Hilb^n(\C^2)\cong \Proj\left(\bigoplus_{d\geq 0} A^d\right)
\end{equation}
and 
\begin{equation}
X_n\cong \Proj\left(\bigoplus_{d\geq 0} J^d\right).
\end{equation}
\end{theorem}
\begin{proof}
See \cite[Proposition 2.6]{Haiqt}.
\end{proof}
\begin{corollary}
\label{cor:trighilb}
We have
\begin{equation}
\Hilb^n(\C^*\times \C)\cong \Proj\left(\bigoplus_{d\geq 0} A^d_\bx\right)
\end{equation}
and
\begin{equation}
Y_n\cong \Proj\left(\bigoplus_{d\geq 0} J^d_\bx\right),
\end{equation}
where the subscript $\bx$ denotes localization in the $x_i$.
\end{corollary}
\begin{proof}
Both of these equations describe blow-ups; the first along the diagonals in $(\C^*\times\C)^n/\Sn$ and the second along the diagonals
in $(\C^*\times \C)^n$. Note that $(J^{(d)})_\bx=J^{(d)}_\bx$ since localization commutes with intersection. Since blowing up commutes with restriction to open subsets \cite[Lemma 30.30.3]{StacksProject}, Theorem \ref{thm: blow-up description hilb} gives the result.
\end{proof}

There are several relevant sheaves on $\Hilb^n(\C^2)$ and $X_n$ that 
relate to $H_*^T(\Sp_\gamma)$ and $H_*^T(\tSp_\gamma)$ naturally. From 
the Proj construction we naturally get very ample line bundles $\cO_?(1)$ on both $?=X_n$ and $?=\Hilb^n(\C^2)$. Note that it is immediate 
from the construction that $$\cO_{X_n}(1)=\rho^*\cO_{\Hilb^n(\C^2)}(1).
$$ On $\Hilb^n(\C^2)$ there is also a tautological rank $n$ bundle $\cT
$ whose fiber at $I$ is given by $\C[\bx,\by]/I$. Its determinant 
bundle can be shown to equal $\cO(1)$.

As noted before, $\Hilb^n(\C^*\times \C)$ is the preimage under the Hilbert-Chow map of $(\C^*\times \C)^n/\Sn$, it is a (Zariski) open subset of $\Hilb^n(\C^2)$. Similarly, $Y_n=\rho^{-1}(\Hilb^n(\C^*\times\C))\subset X_n$ is an open subset.
Restriction then gives very ample line bundles $$\cO_{Y_n}(1)=\cO_{X_n}(1)|_{Y_n},\;
\cO_{\Hilb^n(\C^*\times\C)}(1)=\cO_{\Hilb^n(\C^2)}(1)|_{\Hilb^n(\C^*\times\C)}.$$

\begin{definition}
\label{def:procesi}
Let $\cO_{X_n}$ be the structure sheaf of the isospectral Hilbert scheme. Define the {\em Procesi bundle} $\cP:=\rho_*\cO_X$  on $\Hilb^n(\C^2)$.
\end{definition}
In particular, $H^0(\Hilb^n(\C^2),\cP\otimes \cO(d))=J^d$.
\begin{theorem}[The $n!$ theorem, \cite{Hai1}]
The Procesi bundle is locally free of rank $n!$ on $\Hilb^n(\C^2)$.
\end{theorem}
Localizing the ideal $J$ at $\bx$, we get the following result.
\begin{proposition}
\label{prop:procesi}
Let $\gamma=zt^d\in \fg\fl_n\otimes \cK$ as before. Then
\begin{equation}
H^0(\Hilb^n(\C\times \C^*),\cP\otimes \cO(d))=J^{(d)}_\bx\cong \Delta^dH_*^T(\Sp_\gamma).
\end{equation}
\end{proposition}
\begin{proof}
We have by definition that
$$H^0(\Hilb^n(\C\times \C^*),\cP\otimes \cO(d))=H^0(Y_n,\cO_{Y_n}(d)).$$ Since $Y_n\subset X_n$ is in fact a principal open subset determined by $\prod_{i=1}^n x_i\in \C[\bx^\pm, \by]^{\Sn}$, restriction to the open subset coincides with localization. So we get 
$$H^0(Y_n,\cO_{Y_n}(d))=J_\bx^{(d)}.$$
By Theorem \ref{thm:mainthm}, we conclude $$H^0(\cP\otimes \cO(d), \Hilb^n(\C^*\times \C))\cong \Delta^dH_*^T(\Sp_\gamma).$$
\end{proof}

Although it is not clear to us what the cohomology of the affine Springer fiber $\tSp_\gamma$ in $\Fl_G$ describes in these terms, we make the following conjecture.
\begin{conjecture}
\label{conj:affineflag}
As graded $\C[y_1,\ldots,y_n]$-modules, we have 
\begin{equation}
H^0(\cP\otimes \cP^*\otimes \cO(d), \Hilb^n(\C^*\times \C))\cong
\Delta^d \cdot H_*^T(\tSp_{zt^d}).
\end{equation}
\end{conjecture}
\begin{example}
When $d=0$, the above conjecture states $$H^0(\cP\otimes \cP^*,
\Hilb^n(\C^*\times \C))=\C[\tW]\otimes \C[\by]=\C[\bx^\pm, \by]\rtimes W\cong H_*^T(\tSp_{z}).$$
If it is also true for $d=1$, Theorem \ref{thm:affineflagmain} implies that $$H^0(\cP\otimes \cP^*\otimes \cO(1), \Hilb^n(\C^*\times \C))\cong \tJ_{GL_n}.$$
\end{example}
\begin{remark}
The motivation for Conjecture \ref{conj:affineflag} is as follows. In \cite{GoSt}, Gordon and Stafford relate $J_n^{(d)}$ and the Procesi bundle to the spherical representation of the trigonometric DAHA in type A. For $d=1$, the antisymmetrized version of this representation has the same size (as an $S_n$-representation) as $\cP\otimes \cP$, as does $H_*^T(\tSp_{tz})$. Since 
$H_*^T(\tSp_{tz})$ also carries a trigonometric DAHA-action (at $c=0$) by results of Oblomkov-Yun \cite{OY1}, it is plausible to conjecture that it is ''the same" module as the Gordon-Stafford construction would give.\end{remark}

\subsection{Diagonal coinvariants and a conjecture on the lattice action}
\label{sec:diagonalcoinvariants}
When $G=GL_n$, 
it is known that the fibers of the Procesi bundle $\cP$, as introduced in the previous section, at torus-fixed 
points in $\Hilb^n(\C^2)$ afford the regular representation of $\Sn$ \cite{Hai1}, and in particular have dimension $n!$. On the other hand, 
they appear as quotients of the ring of {\em diagonal coinvariants} (sometimes also called diagonal harmonics)
$$DR_n:=\C[\bx,\by]/\C[\bx,\by]_+^{\Sn},$$ which is now known to be $(n+1)^{n-1}$-
dimensional. Additionally, it is known that the isotypic component  $DH^{sgn}_n$ has dimension $C_n$, where $C_n$ is the $n$th Catalan number, and that its bigraded character is given by $$(e_n,\nabla e_n).$$ Here $(-,-)$ is the Hall inner product on symmetric 
functions over $\Q(q,t)$ and $e_j$ denotes the $j$th elementary symmetric function. The operator $\nabla$ is the nabla operator introduced by Garsia and Bergeron \cite{GaBe}.

As far as the relation with affine Springer theory goes, from work of Oblomkov-Yun, Oblomkov-Carlsson and Varagnolo-Vasserot \cite{OY1}, \cite{CO}, \cite{VV1}, it follows that we have, up to regrading, 
$$H^*(\tSp_{\gamma'})\cong DR_n, H^*(\Sp_{\gamma'})\cong DR_n^{sgn},$$ where $\gamma'$ is an endomorphism of $\cK^n=\text{span}\{e_1,\ldots,e_n\}_{\cK}$ given by $\gamma'(e_i)=e_{i+1}, i=1,\ldots, n-1$ and $\gamma'(e_{n})=te_1$. Note that in this case, $\gamma'$ is elliptic so that $\Sp_{\gamma'}$ and $\tSp_{\gamma'}$ are projective schemes of finite type  and thus their cohomologies are finite dimensional. In fact, after adding some equivariance to the picture the cohomologies in question become the finite-dimensional representations of the trigonometric and rational Cherednik algebras with parameter $c=\frac{n+1}{n}.$

It is a conjecture of Bezrukavnikov-Qi-Shan-Vasserot (private communication) that under the lattice action of $\Lambda$ on 
$H^*(\tSp_\gamma)$, where $\gamma=zt$, we also have 
$$H^*(\tSp_\gamma)^\Lambda\cong DR_n$$ and $$H^*(\Sp_\gamma)^\Lambda\cong DR_n^{sgn}.$$

So far, we are not able to prove this conjecture, but can deduce the sign character part as follows.

\begin{theorem}
\label{thm:bezconj}
We have $$H_*(\Sp_\gamma)_\Lambda \cong DR_n^{sgn}.$$
\end{theorem}
\begin{proof}
Using Theorem \ref{thm:mainthm}, we compute that $$H_*(\Sp_\gamma)\cong\frac{H^T_*(\Sp_\gamma)}{\langle \by \rangle}.$$ As the actions of $\C[\bx^\pm]$ and $\C[\by]$ commute, the result is still a $\C[\bx^\pm]$-module. Taking coinvariants, we have 
$$H_*(\Sp_\gamma)_\Lambda:=\frac{H_*(\Sp_\gamma)}{\langle 1-\bx\rangle H_*(\Sp_\gamma)}\cong\frac{H^T_*(\Sp_\gamma)}{\langle 1-\bx, \by\rangle H^T_*(\Sp_\gamma)}.$$ The last equality follows from the isomorphism theorems for modules. Here $\langle 1-\bx\rangle$ means the set $\{1-x_1,\ldots, 1-x_n\}$ and $\by$ means the set $\{y_1,\ldots, y_n\}$.

On the other hand, $$J_{GL_n}/\langle x_1-1,\ldots, x_n-1, y_1,\ldots, y_n\rangle J_{GL_n}$$ may be identified with $J/\langle x_1-1,\ldots, x_n-1, y_1,\ldots, y_n\rangle J,$ where $$J:=\bigcap_{i\neq j} \langle x_i-x_j,y_i-y_j\rangle\subset \C[\bx,\by]$$ since quotient and localization commute. Since $J$ is translation-invariant with respect to $x_i\mapsto x_i+c, i=1,\ldots, n$, so that $$J/\langle x_1-1,\ldots, x_n-1, y_1,\ldots, y_n\rangle J\cong J/\langle \bx, \by\rangle J.$$ On the other hand, we have that
$J/\langle \bx, \by\rangle J\cong DR_n^{sgn}$ by the fact that the left-hand side is the space of sections of $\cO(1)$ on the zero-fiber of the Hilbert-Chow map inside $\Hilb^n(\C^2)$ \cite[Proposition 6.1.5]{Hai1}. 
\end{proof}
\begin{corollary}
We have $$H^*(\Sp_\gamma)^\Lambda\cong \Hom_{\C}(DR_n^{sgn},\C).$$
\end{corollary}
\begin{proof}
Let $X$ be an equivariantly formal $T$-ind-scheme with a (commuting) action of $\Lambda$. Then we have 
\begin{align*}
&(H^*(X))^\Lambda &\\
\cong&\left(\frac{H_T^*(X)}{\C[\ft]_+H^*_T(X)}\right)^\Lambda &\\
\cong&\left(\Hom_{\C[\ft]}(H^T_*(X),\C[\ft])\otimes_{\C[\ft]} \C \right)^\Lambda &\\
\cong&\left(\Hom_{\C[\ft]}(H^T_*(X),\C)\right)^\Lambda &\\
\cong&\Hom_{\C}(H_*(X),\C)^\Lambda &\\
\cong& \Hom_{\C[\Lambda]}(H_*(X),\C)&\\
\cong& \Hom_{\C}(H_*(X)_\Lambda,\C). & 
\end{align*}
The second isomorphism follows from the fact that whereas $H_*^T(X)$ is defined as the {\em restricted dual} of $H^*_T(X)$ over $\C[\ft]$, the ordinary dual of $H_*^T(X)$ over $\C[\ft]$ is $H^*_T(X)$.
\end{proof}
\begin{remark}
By the above corollary and conjecture, it seems that it is best to think of $H^*(\tSp_\gamma)^\Lambda$ as the (isomorphic) dual space to $DR_n$, called the {\em diagonal harmonics}, that can be described also as 
$f\in \C[\bx,\by]$ annihilated by all $P\in \C[\partial_{x_1},\ldots,\partial_{x_n},\partial_{y_1},\ldots,\partial_{y_n}]_+^{S_n}$.
\end{remark}
\begin{corollary}
One has 
$$\dim_{q,t} H_*(\Sp_\gamma)_\Lambda=\langle e_n, \nabla e_n\rangle,$$
and $\dim_\C H_*(\Sp_\gamma)_\Lambda=C_n$, where $C_n$ is the $n$th Catalan number.
\end{corollary}
\begin{remark}
In the spirit of Conjecture \ref{conj:affineflag}, it seems likely that the approach from above can be used to show that $H_*(\tSp_\gamma)_\Lambda \cong DR_n$. Both would follow from an explicit description of $H^0(\cP\otimes \cP,\Hilb^n(\C^2))$.
\end{remark}

\subsection{Rational and elliptic versions}
\label{sec:ratell}
We now comment on the relation of our results to $\Hilb^n(\C^2)$ and $\Hilb^n(\C^*\times \C^*)$. These are known to quantize to the full DAHA and the rational Cherednik algebra of $\mathfrak{gl}_n$. Let us start with the elliptic version. In Theorem \ref{thm:mainthmktheory}, the description of the K-homology of $\Sp_\gamma$ is given. As blow-up commutes with restriction to opens, we have the following  analogue to Theorem \ref{thm: blow-up description hilb} and Corollary \ref{cor:trighilb}.
\begin{corollary}
We have 
\begin{equation}\Hilb^n(\C^*\times \C^*)\cong \Proj \left(\bigoplus_d A_{\bx,\by}^d\right)
\end{equation}
and 
\begin{equation}
Y'_n\cong \Proj \left( \bigoplus_d (J')^d\right).
\end{equation}
Here the subscript $\bx, \by$ denotes localization in $\prod x_i$ and $\prod y_i$, and $Y'_n$ is the isospectral Hilbert scheme on $\C^*\times \C^*$.
\end{corollary}
 Analogously to Proposition \ref{prop:procesi}, we have the following.
\begin{proposition}
We have 
\begin{equation}
H^0(\cP\otimes \cO(d), \Hilb^n(\C^*\times \C^*))\cong (\Delta')^dK^T(\Sp_{t^dz})
\end{equation}
\end{proposition}
Let now $\Gr+_{GL_n}:=\bigsqcup_{\lambda \in \Lambda^+} \Gr^\lambda$ be the positive part of the affine Grassmannian. Let $\Sp_{t^d z}\cap \Gr^+_{GL_n}$ Then the $T$-fixed points in both are identified with $\Lambda^+$ and their classes in $\C[\Lambda]$ with the monomials without negative powers. Intersecting $\Delta^d H_*^T(\Sp_{t^dz})$ with $H_*^T(\Lambda^+)$ gives $J^{(d)}\subset \C[\bx,\by]$. From the proof of Theorem \ref{thm:mainthm}, it is not hard to see that this agrees with $\Delta^d H_*^T(\Sp_{t^d z}^+)$. In particular,  we have
\begin{theorem}
\label{thm:rationalprocesi}
$$H^0(\cP\otimes \cO(d), \Hilb^n(\C^2))\cong \Delta^dH_*^T(\Sp_{t^d z}^+).$$
\end{theorem}
\begin{remark}
When $n=2$, it is not hard to see that $\Sp_{tz}^+$ is isomorphic to the Hilbert scheme of points on the curve singularity $\{x^2=y^2\}$, as studied in Section \ref{sec:cptfdjacobians}. In forthcoming work, it will be shown that this is the case for higher $n$ as well.
\end{remark}
\subsection{Other root data}
\label{sec:isospectral-othertypes}
In this section, we consider a general connected reductive group $G$. As we will see, many things from the above discussion are not as straightforward. 

In \cite{Hai1}, Haiman discusses the extension of his $n!$ and $(n+1)^{n-1}$ conjectures to other groups. 
The naturally appearing space here is $T^*\ft$ with its diagonal $W$-
action. 
In the case of a general reductive group, Gordon \cite{Gor} has proved that there is a 
canonically defined doubly graded quotient ring $R^W$ of the coinvariant 
ring $$\C[T^*\ft]/\C[T^*\ft]_+^W$$ whose dimension is $(h+1)^r$ for the 
Coxeter number $h$ and rank $r$. It is also known that $sgn\otimes R^W$ 
affords the permutation representation of $W$ on $Q/(h+1)Q$ for $Q$ the 
root lattice of $G$. It would be interesting to compare the lattice-invariant parts of $H^*(\Sp_\gamma)$ and $H^*(\tSp_\gamma)$ to this quotient in other Cartan-Killing types.

We have now seen how the antisymmetric pieces of spaces of diagonal coinvariants appear from affine Springer fibers in the affine Grassmannian. On the other hand, we have seen that in type $A$, the antisymmetric part of $\C[\bx, \by]$ plays the main role in the construction of the isospectral Hilbert scheme $X_n$ as a blow-up.
From solely the point of view of Weyl group representations, it would be then natural to consider 
the $sgn$-isotypic part of $\C[T^*\ft], \C[T^*T^\vee]$. 

We now restate and prove Theorem \ref{thm:introanti}. 
\begin{theorem}
\label{thm:anti}
Let $I_G\subseteq \C[T^*T^\vee]$ be the ideal generated by $W$-alternating polynomials in $\C[T^*T^\vee]$ with respect to the diagonal action. Then there is an injective map
$$I^d\hookrightarrow J_G^{(d)}=\Delta^dH_*^T(\Sp_\gamma).$$
\end{theorem}
\begin{proof}
Write $(\bx,\by)=(x_1,\ldots,x_r,y_1,\ldots, y_r)$ for the coordinates on $T^*T^\vee$ determined by $x_i=\exp(\epsilon_i)$ and where the $y_i$ are the 
cotangent directions. Let $f(\bx,\by)\in I_G$ and let
$\alpha \in \Phi^+$ be a positive root. Denote by $s_\alpha$ the corresponding reflection. Without loss of 
generality we may take $f(\bx,\by)$ to be $W$-antisymmetric. Then at 
points $(\bx,\by)$ where $\exp(\alpha^\vee)=1$, $\partial_\alpha=0$ we 
must have $s_\alpha\cdot f(\bx,\by)=-f(\bx,\by)=f(\bx,\by)$ for any $s_\alpha$. Thus $f(\bx,\by)=0$ on the subspace arrangement defined by $J_G$, and by the Nullstellensatz $f\in J_G$. Taking $d$th powers and observing that $J^d_G
\subseteq J^{(d)}_G$ for any $d$ gives the result.
\end{proof}

\begin{proposition}
There is a natural graded algebra structure on $$\bigoplus_{d\geq 0} J^{(d)}_G$$
given by multiplication of polynomials:
$$J^{(d_1)}_G\times J^{(d_2)}_G\to J^{(d_1+d_2)}_G.$$
\end{proposition}
\begin{proof}
Suppose $f_i\in \bigcap_{\alpha\in \Phi^+}\langle 1-\alpha^\vee, y_\alpha\rangle^{^i}$, $i=1,2$. Then $f_1f_2\in \langle 1-\alpha^\vee, y_\alpha\rangle^{d_1+d_2}$ for all $\alpha$, so that $J^{(d_1)}_G J^{(d_2)}_G\subseteq J^{(d_1+d_2)}_G$.
\end{proof}
The following Theorem was communicated to the author by Mark Haiman.
\begin{theorem}
$$Y_G:=\Proj\left(\bigoplus_{d\geq 0} J^{(d)}_G\right)$$
is a normal variety.
\end{theorem}
\begin{proof}
The  powers  of  an  ideal  generated  by  a  regular sequence are integrally closed, as is an intersection of integrally closed ideals. Therefore, each of the ideals $J^{(d)}_G$ is integrally closed, and so is the algebra $$\bigoplus_{d\geq 0} J^{(d)}_G.$$
By construction, the ring is an integral domain, so $Y_G$ is by definition normal. See also \cite[Proposition 3.8.4]{Hai1} for the proof of this statement in type A.
\end{proof}
\begin{remark}
This $\Proj$-construction is sometimes called the symbolic blow-up. Since we do not know if $J^d_G=J^{(d)}_G$, and likely this is not the 
case, the ring $\bigoplus_{d\geq 0} J^{(d)}_G$ is not generated in 
degree one. However, if we did have translation invariance in the $
\Lambda$-direction in this case, we could deduce results about the 
geometry of the double Coxeter arrangement in $T^*\ft^\vee$ by similar arguments as in type A. It would 
be reasonable to suspect $Y_G$
also has a map to the ``$W$-Hilbert scheme" or some crepant resolution but we do not discuss these possibilities any further.
It should be mentioned that in \cite{Gin}, Ginzburg studies the ``isospectral commuting variety". He has proved that its normalization  is Cohen-Macaulay and Gorenstein. It would be interesting to know how this variety relates to the variety $Y_G$.
\end{remark}

\section{Relation to knot homology}
\label{sec:braids}
Gorsky and Hogancamp have recently defined $y$-ified Khovanov-Rozansky 
homology $\HY(-)$ \cite{GoHo}. It is a deformation of the triply-graded knot homology theory of Khovanov and Rozansky \cite{KR2}, which is often dubbed HOMFLY homology, for it categorifies the HOMFLY polynomial. In this section, we discuss the relationship of the results in previous sections to these link homology theories.

Recall that the HOMFLY homology of a braid closure $\bar{\beta}$ can be defined \cite{KR2} as the Hochschild homology of a certain complex of Soergel bimodules called the Rouquier complex. We denote the triply graded homology of 
$\bar{\beta}$ by $\HHH(\bar{\beta})$.

As stated above, there exists a nontrivial deformation of this theory, called $y$-ification, which takes place in an enlarged category of curved complexes of $y$-ified ``Soergel bimodules". It was defined in \cite{GoHo} and in practice is still defined as the Hocschild homology of a deformed Rouquier complex. We denote the $y$-ified homology groups of a braid closure $L=\bar{\beta}\subset S^3$ by $\HY(L)$. They are triply graded modules over a superpolynomial ring $\C[x_1,\ldots, x_m,y_1,\ldots, y_m,\theta_1,\ldots, \theta_m]$, where $m$ is the number of components in $L$. The $\theta$-grading comes from Hochschild homology, and we will mainly be interested in the Hochschild degree zero part. We will denote this by $\HY(L)^{a=0}$. See \cite[Definition 3.4]{GoHo} for the precise definitions.

\begin{definition}
Let $\cox_n\in \Br_n$ be the positive lift of the Coxeter element of $\Sn$. The $d$th power of the {\em full twist} is the braid $\FT_n^d:=\cox_n^{nd}$.
\end{definition}
\begin{remark}
The element $\FT_n$ is a central element in the braid group and it is known to generate the center.
\end{remark}

\begin{theorem}[\cite{GoHo}]
We have $\HY(\FT_n^d)^{a=0}\cong J^d\subset \C[\bx,\by]$.
\end{theorem}
\begin{corollary}
\label{coro:ft}
There is an isomorphism of $\C[\bx^\pm, \by]$-modules 
$$\Delta^dH_*^T(\Sp_\gamma)\cong \HY(\FT_n^d)^{a=0}\otimes_{\C[\bx]}\C[\bx^\pm]$$ for $\gamma=zt^d$.
\end{corollary}

Following Theorem \ref{thm:mainthm} for $G=GL_n$, it is interesting to consider the homologies of the powers of the full twist as $d\to \infty$. 
By \cite{Hogancamp}, it is known that the $a=0$ part of the ordinary HOMFLY homology of $\FT^\infty_n$ is given by a polynomial ring on generators $g_1,\ldots, g_n$ of degrees $1,\ldots, n$, which coincide with the exponents of $G$, and in particular with the equivariant BM homology of the affine Grassmannian. In the context of {\it loc. cit.} the corresponding algebra appears as the endomorphism algebra of a categorified Jones-Wenzl projector. The corresponding statement in $y$-ified homology is stronger, and states 
$$\HY(\FT_n^\infty)\cong\C[g_1,\ldots, g_n, y_1,\ldots, y_n]$$
as $\C[\by]$-modules.

\begin{theorem}
Consider the system of inclusions $$H_*^T(\Sp_{t^dz})\to H_*^T(\Sp_{t^{d+1}z}).$$ Taking the colimit in the category of $\C[\bx^\pm,\by]$-modules, we have
$$H^T_*(\Gr_{GL_n})\cong \HY(\FT_n^\infty)\cong \C[g_1,\ldots,g_n,y_1,\ldots,y_n].$$ In particular, the lattice action is trivial. 
\end{theorem}
\begin{remark}
Note that this looks like the coordinate ring of the open affine where the points on the (isospectral) Hilbert scheme have distinct $x$-coordinates by \cite[Section 3.6]{Hai1}. However, it does not seem to be true that the algebra structure matches (it does on cohomology). Namely, the algebra structure on $H_*^T(\Gr_{GL_n})$ is that of the ''Peterson subalgebra" studied by various authors, but this does not agree with the algebra structure of $\HY(\FT^\infty_n)$ found by Gorsky and Hogancamp.
On the other hand, one expects some relation of $$\varinjlim \Delta^d H^T_*(\Sp_\gamma),$$ where the system of maps is given by multiplication by $\Delta$,  to the categorified Jones-Wenzl projector for the one-column partition.
\end{remark}

We record the following theorem from \cite[Theorem 1.14]{GoHo}, relating commutative algebra in $2n$ variables to the link-splitting properties of $\HY(-)$.
\begin{theorem}
\label{thm:splitting}
Suppose that a link $L$ can be transformed to a link $L'$ by a sequence 
of crossing changes between different components. Then there is a 
homogeneous ``link splitting map" $$\Psi: \HY(L)\to \HY(L')$$ which 
preserves the $\Q[\bx,\by,\mathbf{\theta}]$-module structure. If, in addition, $HY(L)$ is free as a $\Q[\by]$-module, then $\Psi$ is injective. If the crossing changes only involve components i and j, then the link splitting map becomes a homotopy equivalence after inverting $y_i-y_j$, where $i,j$ label the components involved.
\end{theorem}

The cohomological purity of $\Sp_\gamma$ should be compared to the parity statements in \cite[Definitions 1.16, 3.18, 4.9]{GoHo}. Namely, we have the following Theorem.

\begin{theorem}[\cite{GoHo}, Theorem 1.17]
If an $r$-component link $L$ is parity then $$\HY(L)\cong \HHH(L)\otimes \C[\by]$$ is a
free $\C[\by]$-module.

In particular, $\HY(L)/\by \HY(L)\cong \HHH(L)$ as triply graded vector spaces.

Consequently any link splitting map identifies $\HY(L)$ with a $\Q[\bx,\by,\mathbf{\theta}]$-
submodule of $\HY(\text{split}(L))$.
\end{theorem}

In the case of the powers of the full twist, Theorem \ref{thm:splitting} is easy to understand. Namely, inverting $y_i-y_j$ we simply remove the ideal $(x_i-x_j,y_i-y_j)$ from the intersection $J$.
This also clearly holds for $J^{(m)}$.
Let us consider similar properties for the anti-invariants, following Haiman \cite{HaiMacdonald}.
\begin{lemma}
The ideal $I$ factorizes locally as the product of $I$ for parabolic subgroups of $\Sn$.

\end{lemma}
\begin{proof}
Let $g$ be a generator of $$I'=I(x_1,y_1,\ldots,x_r,y_r)I(x_{r+1},y_{r+1},\ldots, x_n,y_n),$$ alternating in the first $r$ and last $n-r$ 
indices.
Let $h$ be any polynomial which belongs to the localization $J_Q$ at 
every point $Q\neq P$ in the $\Sn$-orbit of $P$, but doesn't vanish at 
$P$. Then $f=\text{Alt}(gh)$ belongs to $I$. The terms of $f$ corresponding to $w\in \Sn$ not stabilizing $P$ belong to $J_P$, by construction of $h$. Since $g$ alternates with 
respect to the stabilizer of $P$, the remaining terms sum to a unit 
times $g$, or more precisely $g\sum_{wP=P}wh$. Hence $g\in I_P$. This means that  $I$ and $I^m$  factorize 
locally as products of the corresponding ideals in the first $r$ and 
last $n-r$ indices. 
\end{proof}
It is curious to note that a similar property  holds for the affine 
Springer fibers.
As shown in \cite[Theorem 10.2]{GKM2}, we have the following 
relationship between equivariant (co)homology of $\Sp_\gamma$ and the 
corresponding affine Springer fiber of an ``endoscopic" group. This is 
to say, $G'$ has a maximal torus isomorphic to $T$ and its roots with 
respect to this torus can be identified with a subset of $\Phi(G,T)$. 
If $G'$ is such a group for $G=GL_n$ (which in this case can just be 
identified with a subgroup of $G$), we have an isomorphism
\begin{equation}\label{eqn-FL0}
H_i^T(\Sp_{\gamma}; \C)_S \cong H_{i-2r}^{T}
(X^T_{\gamma_T}; \C)_S,
  \end{equation}
where $S$ is the multiplicative subset generated by $(1-\alpha^{\vee})$, where the coroots $\alpha^\vee$ run over all coroots {\em not} corresponding to 
$G'$. If we denote this set by $\Phi(G)^+-\Phi(G')^+$, then $r$ is the 
cardinality of this finite set times $d$. For general diagonal $\gamma
$, or alternatively the pure braids discussed in the introduction, $r$ 
is the degree of the corresponding product of Vandermonde determinants, 
or in the automorphic form terminology the homological transfer factor.
The fact that this localization corresponds exactly to link splitting 
in $y$-ified homology (after using the Langlands duality $\bx
\leftrightarrow \by$) is in the author's opinion quite beautiful and 
deep.

\section{Hilbert schemes of points on planar curves}
\label{sec:cptfdjacobians}
\subsection{Hilbert schemes on curves and compactified Jacobians}
In the case $G=GL_n$, which we will assume to be in from now on,
the affine Grassmannian has a description as the space of lattices:
$$G(\cK)/G(\cO)=\{\Lambda\subseteq \cK^n|\Lambda\otimes_\cO\cK=\cK^n, \Lambda \text{ a projective } \cO^n \text{-module}\}$$
We may think of $\Sp_\gamma$ as $\{\Lambda|\gamma\Lambda\subseteq \Lambda\}$. If $\gamma$ is regular semisimple, the characteristic polynomial of $\gamma$ determines a polynomial $P_\gamma(x)$ in $\cO[x]$, which equals the minimal polynomial of $\gamma$. Denote
$A=\cO[x]/P_\gamma(x)$, $F=\text{Frac}(A)$. As a vector space, we then have $F=\cK[x]/P_\gamma(x)\cong \cK^n$, and $\Sp_\gamma$ can be identified with the space of fractional ideals in $F$. On the other hand, this is by definition the Picard factor or local compactified Picard associated to the germ $\cO[[x]]/P_\gamma(x)$ of the plane curve $C=\{P_\gamma(x)=0\}$ \cite{AIK}.

By eg. Ng\^o's product theorem \cite{Ngo}, there is a homeomorphism of stacks 
$$\bPic(C)\cong \Pic(C)\times^{\prod_{x\in C^{sing}\Pic(C_x)}}\prod_{x\in C^{sing}}\bPic(C_x).$$

Call $\gamma$ elliptic if it has anisotropic centralizer over $\cK$, or equivalently $P_\gamma(x)$ is irreducible over $\cK$. There has been a lot of work in determining the compactified Jacobians of $C$, in particular in the cases where $P_\gamma(x)=t^n-x^m$, $\gcd(m,n)=1$ \cite{OY2, LS, Piont, GM1}.

There is always an Abel-Jacobi map $AJ: C^{[n]}\to \bPic(C)$ given by $\cI_Z\mapsto \cI_Z$. It is known that 
for elliptic $\gamma$ this becomes a $\P^{n-2g}$-bundle over $\bPic^n(C)$ for $n>2g$. For nonelliptic $\gamma$ as we are interested in, there is no such stabilization.
On the local factors it is known $AJ$ is an isomorphism for $n>2g$ with $\bPic^n(C_0)$ for the elliptic case, and in the nonelliptic case it is known that $AJ$ is a dominant map to a union of irreducible components in the same connected component of $\bPic(C_0)$.

The precise homological relation between $\bPic(C)$ and $\Hilb^n(C)$ is most concisely summarized in the following Theorem of Maulik and Yun \cite[Theorem 3.11]{MY}.
\begin{theorem}
\label{thm:localcptfd}
Let $\hat{\cO}$ be a planar complete local reduced $k$-algebra of dimension one, with $r$ analytic branches. Assume $\text{char }k=0$ or $\text{char }k>\text{mult}_0(\hat{\cO})$. Then there is a filtration
$P_{\leq i}$ on $H^*(\bPic(\hat{\cO})/\Lambda)$ so that we have the following identity in $\Z[[q,t]]$:
$$\sum_n \sum_{i,j}(-1)^j\dim \Gr_i^WH^j(X)q^it^n=\frac{\sum_i \sum_{k,j}(-1)^j\Gr_i^P\Gr_k^WH^j(\bPic(\hat{\cO})/\Lambda)q^kt^i}{(1-t)^r}.$$
 Here $W_{\leq k}$ is the weight filtration.
\end{theorem}

In addition to the relationship of $C^{[n]}$ with the compactified Jacobians, conjectures of Oblomkov-Rasmussen-Shende \cite{OS, ORS} predict 
that they in fact determine the knot homologies of the links of singularities of $C$ and vice versa. For simplicity, assume $C$ has a unique singularity at zero, and let $C_0^{[n]}$ be the punctual Hilbert scheme of subschemes of length $n$ in $C$ supported at zero.

 Then \cite[Conjecture 2]{ORS} states
\begin{conjecture}
$$V_0:=\bigoplus_{n\geq 0} H_*(C^{[n]}_0)\cong \HHH^{a=0}(L).$$ 
\end{conjecture}
\begin{remark}
On the level of Euler characteristics, this is known to be true by \cite{Mau}.
\end{remark}
We should mention that there is yet another reason to care about $C^{[n]}$; the Hilbert schemes and their Euler characteristic generating functions are closely related to BPS/DT invariants as shown in \cite{Pand, PT}. In \cite{Pand} some of the examples we are interested in are studied.

In earlier work \cite{Kiv}, the author considered the Hilbert schemes of points on reducible, reduced planar curves $C/\C$.
The main result in {\it loc. cit} is as follows.

\begin{theorem}[\cite{Kiv}, Theorem 1.1]
\label{thm:algebra}
If $C=\bigcup_{i=1}^mC_i$ is a decomposition of $C$ into irreducible
components, the space $V=\bigoplus_{n\geq 0}H_*(C^{[n]},\Q)$ carries a bigraded action of the algebra $$A=A_m:=\Q[x_1,\ldots, x_m, \partial_{y_1},\ldots, \partial_{y_m}, \sum_{i=1}^m y_i, \sum_{i=1}^m \partial_{x_i}],$$ where $V=\bigoplus_{n,d\geq 0} V_{n,d}$ is graded by number of points $n$ and homological degree $d$. Moreover, the operators $x_i$ have degree $(1,0)$ and the operators $\partial_{y_i}$ have degree $(-1,-2)$ in this bigrading. In effect, the operator $\sum y_i$ has degree $(1,2)$ and the operator $\sum \partial_{x_i}$ has degree $(-1,0)$.
\end{theorem}
\begin{example}
\label{ex:node}
In the case $x^2=y^2$,
we have 
$$V=\frac{\C[x_1,x_2,y_1,y_2]}{(x_1-x_2)\C[x_1,x_2,y_1+y_2]}$$ as $\C[x_1,x_2,y_1+y_2,\partial_{x_1}+\partial_{x_2},\partial_{y_1},\partial_{y_2}]$-modules.
\end{example}

\subsection{Conjectural description in the case $C=\{x^n=y^{dn}\}$}
As discussed in the introduction, the representation in Example \ref{ex:node} very similar to the main result in \cite{GKM2} when $G=GL_2$ and $d=1$. 
We now recall said theorem.
\begin{theorem}
\label{thm:gkmmain}
Let $G$ be a connected reductive group and $\gamma=zt^d$ as before. Then the ordinary (i.e. not Borel-Moore) $T$-equivariant homology of $\Sp_\gamma$ is a $\C[\Lambda]\otimes \C[\ft^*]$-module, where $\ft$ acts by derivations, and 
$$H_{*,ord}^T(\Sp_\gamma)\cong\frac{\C[\Lambda]\otimes \C[\ft]}{\sum_{\alpha\in\Phi^+}\sum_{k=1}^d(1-x^{\alpha^\vee})^k\C[\Lambda]\otimes \ker(\partial_\alpha^k)}.$$
\end{theorem}
\begin{example}
If $G=GL_2$, $d=1$, we have
$$H_{*,ord}^T(\Sp_\gamma)=\frac{\C[x_1^\pm,x_2^\pm,y_1,y_2]}{(1-x_1x_2^{-1})\C[x_1^\pm,x_2^\pm,y_1+y_2]}.$$
\end{example}

The above examples, as well as Examples \ref{ex:tacnode}, \ref{ex:three} and Theorem \ref{thm:mainthm} motivate us to conjecture the following.
\begin{conjecture}\label{conjecture}
Let $C=\overline{\{x^n=y^{dn}\}}$ be the compactification with unique singularity and rational components of the curve defined by the affine equation $\{x^n=y^{dn}\}$. Then as a bigraded $A_n$-module (see Theorem \ref{thm:algebra}), we have 
\begin{equation}\label{eq:conjecture}
V:=\bigoplus_{n\geq 0} H_*(C^{[n]},\Q)\cong 
\frac{\Q[x_1,\ldots, x_n,y_1,\ldots, y_n]}{\sum_{i\neq j}\sum_{k=1}^{d} (x_i-x_j)^k\otimes \ker(\partial_{y_i}-\partial_{y_j})^k}.
\end{equation}

\end{conjecture}
\begin{remark}
In some sense, passing from the equivariant homology of affine Springer 
fibers to the Borel-Moore version involves only half of the variables, namely the equivariant 
parameters. It is not immediate from the construction of the $A_m$-
action in \cite{Kiv} what the analogous procedure would be to pass to $H^*(C^{[n]})$ from $H_*(C^{[n]})$. It would be interesting to know, at 
least on the level of bigraded Poincar\'e series, how to compare $V$ to 
the ideal $J^d\subset \C[\bx,\by]$, assuming that Conjecture 
\ref{conjecture} is true. The $q,t$-character of $J^d$ is by work of 
Haiman \cite{Hai1} known to be given by the following inner product of 
symmetric functions: $$\dim_{q,t} J^d=(\nabla^d p_1^n, e_n).$$ Thanks 
to work of Gorsky and Hogancamp \cite{GoHo} we then also know that (up to 
regrading) the bigraded character of $HY^{a=0}(T(n,dn))$ is given by 
the same formula.
\end{remark}

For some support for the conjecture, let us consider the following examples. 

\begin{theorem}[\cite{Kiv}]
\label{thm:node}
When $C=\overline{\{x^2=y^2\}}$, we have that 
\begin{equation}
V=\bigoplus_{n\geq 0} H_*(C^{[n]})\cong \frac{\C[x_1,x_2,y_1,y_2]}{\C[x_1,x_2,y_1+y_2](x_1-x_2)}
\end{equation}
as an $A_2$-module, where $$A_2=\C[x_1,x_2,\partial_{x_1}+\partial_{x_2},y_1+y_2,\partial_{y_1},\partial_{y_2}]\subset \text{Weyl}(\A^{4}).$$
\end{theorem}
\begin{remark}
Note that we get an extremely similar looking result for $H_*^H(\Sp_{\diag(t,-t)})$ and $H_*(C^\bullet)$, where $C^\bullet=\bigsqcup_{n\geq 0} C^{[n]}$ is the Hilbert scheme of points on the curve $C=\{x^2=y^2\}\subset \P^2$.
\end{remark}
\begin{remark}
We are no longer using equivariant homology, but have replaced the equivariant parameters by the fundamental classes of the components of the global curve $C$. It does make sense to consider the equivariant cohomology for the Hilbert schemes of points on $C=\{x^n=y^{dn}\}$, but we do not know how to produce a nice action of a rank $n$ torus in this case and whether it would agree with expectations. Note that there is a natural $(\C^*)^2$-action on $C$ and its Hilbert schemes, coming from the $(\C^*)^2$-action with weights $(d,1)$ on the plane.
\end{remark}

\begin{remark}
In general, we may describe the Hilbert schemes $C^{[2]}$ explicitly for $C=\overline{\{x^n=y^{dn}\}}$. Fix a decomposition into irreducible components $C=\bigcup_{i=1}^n C_i$. Since $C$ has $n$ rational components, there is a component $M_i\cong \Sym^2\P^1\cong \P^2$ for each $i$, and for each $i<j$ we have a component $N_{ij}\cong \text{Bl}_{pt}(\P^1\times \P^1)$, see \cite[Example 5.9]{Kiv}. The $\binom{n}{2}$ components $N_{ij}$ all intersect along an exceptional $\P^1$ that can be identified with $\Hilb^2(\C^2,0)$. Denote this line by $E$. We have $M_i\cap M_j=\emptyset$ for all $i\neq j$, and $M_i\cap N_{jk}\cong \P^1$ if $i=j$ or $i=k$, and $M_i\cap N_{jk}=\emptyset$ otherwise. Denote these lines of intersection by $L_i$. It is helpful to picture them as naturally isomorphic to $C_i$. The $L_i$ do not intersect each other, but intersect $\Hilb^2(\C^2,0)$ at points corresponding to the slopes of the corresponding lines $C_i$.

The homology of $C^{[2]}$ in degree two is spanned by $[L_i], i=1,\ldots, n$ and $E$. Denote the fundamental class $[C_i]\in H_2(C^{[1]})$ by $y_i$. Using the $A_n$-action, we have elements 
$$x_iy_i=[L_i]\in H_2(C^{[2]}), i=1,\ldots, n,\; \text{ and }  x_iy_j=[L_j]-[E], i\neq j.$$
Hence we have the relations
\begin{align*}
(x_i-x_j)(y_i+y_j)=0 & \; \forall i, j \\
(x_i-x_j)y_k=0 & \; k\neq i, j.
\end{align*}
Using these relations, we may express all the classes $[L_i], i=1,\ldots, n$ and $[E]$ as linear combinations of $x_iy_i$ and for example $x_1y_2$. Since $$\dim_\C H_2(C^{[2]})=n+1,$$ there cannot be any other relations in this degree.
This verifies equation \eqref{eq:conjecture} of Conjecture \ref{conjecture} in degree $q^2t^2$.
\end{remark}

\subsection{Compactified Jacobians and the MSV formula}
Homologically, we have the following sheaf-theoretic relationship, along the lines of Theorem \ref{thm:localcptfd}, between the cohomology of the compactified Jacobians and the Hilbert schemes of points $C^{[n]}$,  proved independently by Maulik-Yun and Migliorini-Shende.

\begin{theorem}[\cite{MY,MS}]
Let $\pi: \cC\to B$ be a locally versal deformation of $C$, and $\pi^{[n]}: \cC^{[n]}\to B, \pi^J: \bJac(\cC)\to B$ be the relative Hilbert schemes of points and compactified Jacobians of $\pi$. Then, inside $D_c^b(B)[[q]]$, we have
$$\bigoplus_{n\geq 0} q^n R\pi_*^{[n]}\C=\frac{\oplus q^i ~^pR^i\pi^J_*\C}{(1-q)(1-q\L)},$$ where $\L$ is the Lefschetz motive (ie. the constant local system on $B$ in this case.)
\end{theorem}

For reducible curves, the bigraded structure can be also computed from the theorem of Migliorini-Shende-Viviani  \cite[Theorem 1.16]{MSV}.
\begin{theorem}\label{thm:msvthm}
Let $\{C_S\to B_S\}_{S\subset [m]}$ be an independently broken family of reduced planar curves (see \cite{MSV} for the definition), such that all the $C_S\to B_S$ are H-smooth, ie. their relative Hilbert schemes of points have smooth total spaces, and such that the families $C_S\to B_S$ admit fine compactified Jacobians $\overline{J(C_S)}\to B_S$. Then, inside $D_c^b(\bigsqcup B_S)[[q]]$, we have:
    \begin{align}\label{eq:msvformula}
        (q\L)^{1-g}\bigoplus_{n\geq 0} q^n R\pi_*^{[n]}\C &= \text{Exp}\left((q\L)^{1-g}\frac{\bigoplus q^i IC(\Lambda^iR^1\pi_{sm*}\C[-i])}{(1-q)(1-q\L)}\right)\\
         & = \text{Exp}\left((q\L)^{1-g}\frac{\bigoplus q^i ~^pR^i\pi^J_*\C}{(1-q)(1-q\L)}\right).
    \end{align}
Here, $g: B_S\to \N$ is the upper semicontinuous function giving the arithmetic genus of the fibers, and $\L$ is the Lefschetz motive.
\end{theorem}
\begin{remark}
Later, we will use the substitution $\L\mapsto t^2$, which recovers the Poincar\'e polynomial.
\end{remark}

We turn to a more complicated example of $C^{[n]}$.
\begin{example}\label{ex:three}
Consider the (projective completion with unique singularity of the) curve $\{x^3=y^3\}$, ie. three lines on a projective plane intersecting at a point.

We are interested in computing the stalk of the left hand side of \eqref{eq:msvformula} at the central fiber. On the right, the 
exponential map is a sum over all distinct decompositions of $C=C_1\cup 
C_2\cup C_3$ into subcurves. By symmetry, there are only three 
fundamentally different ones: the decomposition into three disjoint 
lines, the decomposition into a node and a line, and the trivial 
decomposition. Since we know that the fine compactified Jacobians of 
nodes and lines are points \cite{MSV}, these terms on the right hand 
side are relatively easy to compute. Namely, for the three lines we 
have $\left(\frac{q\L}{(1-q)(1-q\L)}\right)^3$, and $\left(\frac{q\L}{(1-q)(1-q\L)}\right)^2$ for the decompositions to a node plus a line.

As to the last term on the right, $C$ has arithmetic genus one, so is 
its own fine compactified Jacobian, as shown by Melo-Rapagnetta-Viviani \cite{MRV1}. Moreover, $C$ can be realized as a type III Kodaira fiber 
in a smooth elliptic surface $f: E\to T$, where $T$ is a smooth curve. 
Let $\Sigma$ be the singular locus of $f$.
By the decomposition theorem of Beilinson-Bernstein-Deligne-Gabber \cite{BBD}, we have from eg. \cite[Example 1.8.4]{dCM} $$Rf_*\Q_E[2]=
\Q_T[2]\oplus (IC(R^1f^{sm}_*\Q_E)\oplus \mathcal{G})\oplus \Q_T$$ 
where $\mathcal{G}$ is a skyscraper sheaf on $\Sigma$ with stalks 
$H_2(f^{-1}(s))/\langle [f^{-1}(s)]\rangle$. Note that the rank of this 
sheaf is the number of irreducible components of the fiber {\em minus 
one}.

The terms in the above direct sum are ordered so that we first have the 
second perverse cohomology sheaf $~^p\mathcal{H}^2(Rf_*\Q_E[2])$, then 
the first one inside the parentheses and lastly the zeroth perverse 
cohomology sheaf. Since the base is smooth $IC(R^1)=R^1$ and its stalk 
is zero at the central fiber. This gives that the numerator of our last 
term is $1+2q\L+q^2\L$. In total, we have 
\begin{align}
    \sum_{n\geq 0} q^nH^*(C^{[n]})&=\left(\frac{q\L}{(1-q)(1-q\L)}\right)^3+3\left(\frac{q\L}{(1-q)(1-q\L)}\right)^2+\\ &\frac{1+2q\L+q^2\L}{(1-q)(1-q\L)}, 
\end{align}
which we compute to be 
\begin{equation}
\frac{q^6\L^3 - 2q^5\L^2 + q^4\L^2 + q^3\L^2 + q^4\L - 2q^3\L + q^2\L + q^2 - 2q + 1}{(1-q)^3(1-q\L)^3} \label{eq:(3,3) poincare}
\end{equation}
\end{example}

Let us now consider the simplest example where $d>1$.
\begin{example}\label{ex:tacnode}
  Similarly, we may consider the projective model of the curve $C=\{x^4=y^2\}$, which has two rational components that are parabolas. This also has arithmetic genus one and by the same line of reasoning as above we have
\begin{align*}
\sum_{n\geq 0} q^nH^*(C^{[n]})&=\left(\frac{q\L}{(1-q)(1-q\L)}\right)^2+ \frac{1+q\L+q^2\L}{(1-q)(1-q\L)}\\&=\frac{q^4\L^2 - q^3\L + q^2\L - q + 1}{(1-q)^2(1-q\L)^2}.
\end{align*}
\end{example}
Let us now compute the Hilbert series, as predicted by Conjecture \ref{conjecture}, in the cases of Examples \ref{ex:three}, \ref{ex:tacnode}. 

\begin{example}
In the case of Example \ref{ex:three}, write 
$$U_i=(x_j-x_k)\C[x_1,x_2,x_3,y_j+y_k,y_i], \; k\neq i\neq j\neq k. $$

Denote by $\gr\dim V$ the $(q,t)$-graded dimension of a bigraded vector space $V$. 
Then \begin{align*}
gr\dim(U_1+U_2+U_3)=&\gr\dim(U_1)+\gr\dim(U_2)+\gr\dim(U_3)\\&-\gr\dim((U_1+U_2)\cap U_3)-\gr\dim(U_1\cap U_2)
\end{align*}
and we compute that:
\begin{align*}
(U_1+U_2)\cap U_3=&(x_1-x_3)\C[x_1,x_2,x_3,y_1+y_2+y_3]\\
&+(x_1-x_2)(x_2-x_3)y_3\C[x_1,x_2,x_3,y_1+y_2+y_3],\\
U_1\cap U_2=&(x_1-x_2)\C[x_1,x_2,x_3,y_1+y_2+y_3].
\end{align*}

We then have
$$\gr\dim(U_1+U_2)\cap U_3=\frac{q+q^4t^2}{(1-q)^3(1-qt^2)}$$ 
and  $$\gr\dim(U_1\cap U_2)=\frac{q^2}{(1-q)^3(1-qt^2)}.$$
Hence 
$$\gr\dim(V)=\frac{1}{(1-q)^3(1-qt^2)^3}-3\frac{q}{(1-q)^3(1-qt^2)^2}+\frac{q+q^2+q^4t^2}{(1-q)^3(1-qt^2)},$$
which can be checked to equal the right-hand side of \eqref{eq:(3,3) poincare}.
\end{example}

\begin{example}
In the case of Example \ref{ex:tacnode}, write 
\begin{align*}
U=&(x_1-x_2)\C[x_1,x_2,y_1+y_2], \\
U'=&(x_1-x_2)^2\left(\C[x_1,x_2,y_1+y_2]\oplus \C[x_1,x_2,y_1+y_2](y_1-y_2)\right).
\end{align*}
Then $U\cap U'=(x_1-x_2)^2\C[x_1,x_2,y_1+y_2]$, and we have that the right hand side of \eqref{eq:conjecture} equals
\begin{align*}
&\frac{1}{(1-q)^2(1-q\L)^2}-\frac{q}{(1-q)^2(1-q\L)^2}-\frac{q^2(1+q\L)}{(1-q)^2(1-q\L)}&\\&+\frac{q^2}{(1-q)^2(1-q\L)}
=\frac{q^4\L^2-q^3\L+q^2\L-q+1}{(1-q)^2(1-q\L)^2}.& 
\end{align*}
\end{example}

As a continuation of Examples \ref{ex:three}, \ref{ex:tacnode}, let us verify that the Poincar\'e series agrees with the Oblomkov-Rasmussen-Shende conjectures in both cases, since this result does not appear in the literature.

\begin{proposition}\label{prop:agrees w eh}
If $C=\{x^3=y^3\}$, then under the substitutions $$q\L\mapsto T^{-1}, \; q\mapsto Q,$$ we have the following equality in $\Z[[q,t]]$:
$$\sum_{n\geq 0} q^n H^*(C_0^{[n]})=f_{000}(Q,0,T),$$
where $f_{000}(Q,A,T)$ denotes the triply graded Poincar\'e series of
$$\HHH(T(3,3)).$$ Note that we are considering the punctual Hilbert schemes $C_0^{[n]}$ here.
\end{proposition}
\begin{proof}
From \cite[page 9]{EH}, we have
\begin{align*}
f_{000}(Q,A,T)= \frac{1+A}{(1-Q)^3}\Big(&(T^3Q^2 + Q^3T^2 - 2T^2Q^2 - 2TQ^3 - 2 QT^3\\ & 
+ T^3 + Q^3 + TQ^2 + QT^2+TQ)+
 (T^2Q^2 \\& - 2TQ^2  - 2QT^2 + T^2 + Q^2 + TQ + T + T)A + A^2\Big).
\end{align*}
It is quickly verified that letting $A=0$ and doing the substitution above gives the result.
\end{proof}

\begin{remark}
In fact, \cite{EH} compute the polynomials $f_v(A,Q,T)$ corresponding to HOMFLY homologies of certain complexes $C_v$, where $v$ is any binary sequence, using a recursive description. All of these complexes are supported in even degree, and it would be interesting to know how the corresponding pure braids are realized as affine Springer fibers. It would also be interesting to understand these recursions either on $\Hilb^n(\C^2)$ or in terms of affine Springer fibers for $GL_n$.
\end{remark}

The case $C=\{x^2=y^4\}$ is slightly more straightforward. $$\sum_{n\geq 0} q^nH^*(C_0^{[n]})=(1-\L^2)^2\sum_{n\geq 0} q^nH^*(C^{[n]})$$
 can be checked to equal with the Poincar\'e polynomial of $\HHH^{a=0}(T(2,4))$ 
for example using \cite[Corollary 15]{ORS}, which states
$$P(\HHH^{a=0}(T(2,4)))=\frac{Q^2+(1-Q)(T^2+QT)}{(1-Q)^2T^2}.$$

\end{document}